\documentclass[DIV=14,letterpaper]{scrartcl}
\pdfoutput=1

\date{}

\usepackage{booktabs,multirow,threeparttable}  

\usepackage{tikz}
\usetikzlibrary{calc}
\usetikzlibrary{matrix}
\usepackage{amsmath,amssymb,amsfonts,amsthm,subfigure}
\usepackage{color}                    
\usepackage[pdftex]{hyperref}

\usepackage{graphicx,graphics}
\usepackage{epstopdf}
\usepackage{enumerate}

\allowdisplaybreaks

\newcommand{\ud}{\mathrm{d}}

\newcommand{\bey}{\begin{eqnarray}}
\newcommand{\eey}{\end{eqnarray}}

\newcommand{\beq}{\begin{equation}}
\newcommand{\eeq}{\end{equation}}
\theoremstyle{plain}
\newtheorem{thm}{\hspace{6mm}Theorem}[section]

\newtheorem{lem}{\hspace{6mm}Lemma\,}[section]

\theoremstyle{definition}
\newtheorem{definition}{\hspace{6mm}Definition}[section]
\theoremstyle{remark}
\newtheorem{example}{\hspace{6mm}Example}[section]
\newtheorem{rem}{\hspace{6mm}Remark}[section]

\title{New superconvergent structures developed from the finite volume element method in 1D\thanks{This work was supported in part by the NSFC under grant \#11701211, the China Postdoctoral Science Foundation under grant \#2017M620106, and the Science Challenge Project grant \#TZ2016002}}

\author{Xiang Wang%
\thanks{School of Mathematics, Jilin University, Changchun 130012, China (wxjldx@jlu.edu.cn).}
\and Junliang Lv%
\thanks{School of Mathematics,Jilin University, Changchun 130012, China. (lvjl@jlu.edu.cn).}
\and Yonghai Li%
\thanks{School of Mathematics, Jilin University, Changchun 130012, China (yonghai@jlu.edu.cn).}
}

\begin{document}
\maketitle

\begin{abstract}
New superconvergent structures are introduced by the finite volume element method (FVEM), which allow us to choose the superconvergent points freely. The general orthogonal condition and the modified M-decomposition (MMD) technique are established to prove the superconvergence properties of the new structures. In addition, the relationships between the orthogonal condition and the convergence properties for the FVE schemes are carried out in Table~\ref{tab:relations}. Numerical results are given to illustrate the theoretical results.

%
\end{abstract}




\section{Introduction}
\label{intro}

The finite volume element method (FVEM) \cite{Bank.1987,Barth.2004,Cai.1990,Cai.1991,Chen.2010, Chen.2015,Chou.2007,Hackbusch.1989,Li.2000,Wang.2019}, which is famous for the local conservation property,
has been studied widely for the stability and $H^1$ estimate
\cite{Chen.2012,Chen.2015,Li.1999,Liebau.1996,Schmidt.1993,Xu.2009,Zhang.2015},
$L^2$ estimate \cite{Chen.2002,Ewing.2002,Huang.1998,Lin.2015,Lv.2010,Lv.2012b,Wang.2016}, and superconvergence \cite{Cao.2013,Cao.2015,Chen.1994b,Lv.2012,Wang.2019b}. In this paper, we mainly focus on the new superconvergent structures developed from the FVEM in 1D. To the authors' knowledge, almost all existing natural superconvergence results of the FEM/FVEM are based on the famous Gauss-Lobatto structure. It's interesting that, the new superconvergent structures introduced in this paper cover the Gauss-Lobatto structure and include much more new FVE schemes.

Superconvergence is the phenomenon that the numerical solution (or the post-processed solution) converges faster than the generally expected rate at certain points or with certain metric.
It is an important issue, which helps to improve the accuracy of numerical methods such as the finite element method (FEM) \cite{Andreev.1988,
Schatz.1996,Thomee.1977,Wahlbin.1995,Wang.2001} and the finite volume element method (FVEM) \cite{Bank.1987,Cai.1990,Cao.2015,Lv.2012,Zhang.2014,Zhang.2015} etc..
The study on superconvergence mainly lies in three aspects:
1) \textbf{the natural superconvergence}, in which the numerical solution superconverges to the exact solution at certain points, such as the famous Gauss-Lobatto structure for the FEM/FVEM, which gives superconvergent points at Gauss points (of the derivative/gradient)
or at Lobatto points (of the function value) (see \cite{Babuska.2007,Cao.2013,Chen.1995,Cockburn.2017,Lin.1996,Wang.2019b,Zhu.1989});
2) \textbf{the global superconvergence}, in which there exists a piecewise $k$-order approximation $u_I$ of $u$, such that we have the estimate $\|u_h-u_I\|_{0}=O(h^{k+2})$ or $|u_h-u_I|_{1}=O(h^{k+1})$ for the numerical solution $u_h$. The global superconvergence
results are the theoretical foundation of the other two types of superconvergence (see \cite{Babuska.2007,Chou.2007,Krizek.1987,Wang.2019b});
3) \textbf{the post-processed superconvergence}, in which the post-processed solution superconverges to the exact solution
in some norm (see \cite{Bank.2003,Lin.2013,Yang.2009}). The superconvergent patch recovery (SPR) \cite{Zienkiewicz.1992}
and polynomial preserving recovery (PPR) \cite{Zhang.2005} are two typical examples of the post-processed superconvergence technique. In this paper, we mainly talk about the first two aspects of superconvergence for the FVEM.

We first propose the general $k$-$r$-order $(k-1\leq r\leq 2k-2)$ orthogonal condition and the modified M-decomposition (MMD) technique for the FVEM, which help to discover and prove the new superconvergent structures. For the $k$-order FVEM, the general $k$-$r$-order orthogonal condition means $r$-order orthogonality
to a polynomial space in the sense of inner product. The dual points as well as the interpolation nodes of the trial-to-test operator are variable, which make it possible to design more FVE schemes for given order $k$. The general $k$-$r$-order orthogonal condition is a generalization of the $k$-$(k-1)$-order orthogonal condition proposed in \cite{Wang.2016}. We still call it the orthogonal condition without causing confusion.
On the other hand, when analyzing the superconvergence of the FEM, it's very technical to find a proper superclose function $u_I$, which bridges the exact solution $u$ and the numerical solution $u_h$. Researchers often decompose the difference $u-u_I$ into a linear combination of the M-polynomials, in which the coefficients of the M-polynomial combinations are obtained through restricting $u-u_I$ for better properties. This method of obtaining the superclose function is called the M-decomposition technique (see \cite{Chen.1995,Cockburn.2017}). The M-decomposition technique works well for the FEM, however it usually fails to be applied to the FVEM directly. For this reason, we propose the modified M-decomposition technique to obtain an appropriate superclose function for the FVEM.

Then, with the help of the orthogonal condition and the MMD technique, we construct and prove the new superconvergent structures for the FVEM. It's shown that, for given $k$-order ($k\geq3$) FVEM, there are much more than one scheme having the superconvergence property. As examples, the relationships between the Gauss-Lobatto structure based and the orthogonal condition based FVE schemes are shown in Figure~\ref{fig:alpha4A5order} (a) for $k=4$ and Figure~\ref{fig:alpha4A5order} (b) for $k=5$.
Furthermore, we also carry out easy ways to construct FVE schemes with superconvergence (see {\em Method I} and {\em Method II} in subsection~\ref{subsec:construction_easy}):
for odd $k$-order FVEM, we can freely choose the $k$ symmetric derivative superconvergent points on any primary element $K$ (excluding the endpoints of $K$);
and for even $k$-order FVEM, we can freely choose the $(k+1)$ symmetric function value superconvergent points on any primary element $K$ (including two endpoints of $K$).


Moreover, for the completeness of the theory, we also present the proof of unconditionally stability, $H^1$ estimates and give the optimal $L^2$ estimates as a side product of the superconvergence of the derivative. Here, we show in Table~\ref{tab:relations} the relationships between the orthogonal condition and the convergence properties for the FVE schemes over symmetric dual meshes:
1) all FVE schemes possess the optimal $H^1$ estimates; 2) the $k$-$(k-1)$-order orthogonal condition ensures the optimal $L^2$ estimates and superconvergence of the derivative (see ``superconv 1'' in Table~\ref{tab:relations}); 3) the $k$-$k$-order orthogonal condition ensures the superconvergence of the function value (see ``superconv 2'' in Table~\ref{tab:relations}).
That is to say, when the $k$-$(k-1)$-order or $k$-$k$-order orthogonal condition is satisfied, the corresponding FVE schemes hold superconvergence properties. We call this superconvergent structure the ``orthogonal structure''.
\begin{table}[htbp!]
  \centering
  \caption{Relations between the orthogonal condition and properties of FVE schemes in 1D}\label{tab:relations}
  \begin{tabular}{c|c|cccc}
  \toprule
  \multicolumn{2}{c}{FVE schemes} & \multicolumn{4}{c}{Properties of FVE schemes}\\
  \cline{1-6}
              &  orthogonal condition & optimal $H^1$ &optimal $L^2$ & superconv 1& superconv 2\\
  \cline{1-6}
  odd order & $k$-$(k-1)$-order & $\surd$     &   $\surd$ & $\surd$ & - \\
    \cline{2-6}
   ($k=2l-1$) & $k$-$k$-order     & $\surd$     &   $\surd$ & $\surd$ & $\surd$ \\
  \cline{1-6}
           & does not satisfy & \multirow{2}{*}{$\surd$}     &  \multirow{2}{*}{-}
                                   & \multirow{2}{*}{-} & \multirow{2}{*}{-} \\
    even order       &    the $k$-$(k-1)$-order    &      &    &  &  \\
    \cline{2-6}
    ($k=2l$) & $k$-$(k-1)$-order  &\multirow{2}{*}{$\surd$} &\multirow{2}{*}{$\surd$}
                                &\multirow{2}{*}{$\surd$} &\multirow{2}{*}{$\surd$}\\
             & ($k$-$k$-order) &      &    &  &  \\
 \bottomrule
  \end{tabular}
  \begin{tablenotes}
      \footnotesize
      \item[1] 1. The ``$\surd$'' mark means holding, while the ``-'' mark means no results.

  \end{tablenotes}
\end{table}

%
%
%
%

Following, we introduce the definition of the FVEM and some notations in section~\ref{sec:prelim}.
Then, the $k$-$r$-order orthogonal condition and the modified M-decomposition (MMD) are discussed in section~\ref{sec:OrthCond}.
In section~\ref{sec:Superconvergence}, we present the superconvergence of the derivative and the function value for FVEM.
In section~\ref{sec:construction}, we carry out the constructions of the FVE schemes with superconvergence.
Finally, we show numerical results in section~\ref{sec:numerical_ex} and make the conclusion in section~\ref{sec:conclusion}.
The stability and $H^1$ error estimate of the FVEM are provided in Appendix A.

%

\section{FVE schemes of arbitrary order} \label{sec:prelim}
Consider the two-point boundary value problem on $\Omega=(0,1)$:
\begin{align} \label{eq:BVP1D}
\left\{ \begin{array}{rcl}
- (p(x) u'(x))'+ q(x) u'(x) + r(x) u(x)= f(x),  \quad \forall x\in\Omega,\\
u(0)=u(1) = 0 ,\qquad\qquad\qquad
\end{array} \right.
\end{align}
where $p\geq p_0>0$, $r-\frac{1}{2}q'\geq \gamma >0$, $p,q,r\in L^{\infty}$ and $f\in L^2(\Omega)$.

\subsection{The trial function space and test function space} \label{subsec:primary_mesh}
\textbf{Primary mesh and trial function space.}
Let $0=x_0<x_1<x_2<\dots<x_{N}=1$ be $N+1$ distinct points on $\overline{\Omega}$. For all $i\in \mathbb{Z}_{N}:=\{1,\dots,N\}$,
we denote $K_{i} = [x_{i-1},\,x_{i}]$, $h_{i}=x_{i}- x_{i-1}$, and $h=\max\limits_{i\in\mathbb{Z}_{N}}h_{i}$. Let
\[
\mathcal{T}_{h}=\{K_{i}: \, i\in \mathbb{Z}_{N} \}
\]
be a partition of $\Omega$, and call it the primary mesh. Choose the $k$-order ($k\geq1$) Lagrange finite element space as the corresponding trial function space
\begin{align*}
\mathit{U}_{h}^{k} := \{ w_{h}\in \mathit{C}(\Omega): \, w_{h}\vert_{K} \in P^{k}(K),\, \forall K\in \mathcal{T}_{h} \quad\!\!\! \mathrm{and}\quad\!\!\! w_{h}|_{\partial \Omega}=0\}.
\end{align*}
Here, $P^{k}(K)$ is the $k$-order polynomial space on $K$. It's clear that $\dim \mathit{U}_{h}^{k}=Nk-1$.



\textbf{Dual mesh and test function space.} 
For $k$-order FVEM, let $-1< G_1 <\dots<G_k<1$ be $k$ symmetric points, to define the dual points, on the reference element $\hat{K}:=[-1,1]$.
There should be $j_0$ parameters $\alpha_{i}$ to locate the $G_{j}$ with
\begin{align}\label{eq:alpha_s}
1>\alpha_{1}>\cdots>\alpha_{j_{0}} >0,  \quad
j_{0} = \left\{
  \begin{array}{ll}
    l-1 , & k=2l-1, \\
    l , & k=2l, \\
\end{array}
\right. \quad l=1,2,3,\dots ,
\end{align}
such that
\begin{align} \label{eq:Gs}
\left\{
  \begin{array}{cc}
    - G_{j} = G_{k-j+1} = \alpha_j, & \mathrm{for}\,\,1\leq j \leq j_{0}, \\
    G_{l} =0, & \mathrm{if}\,\,\,k=2l-1. \\
  \end{array}
\right.
\end{align}

The dual points on each primary element $K_{i}$ ($i\in \mathbb{Z}_{N}$) are defined as the affine transformations of points $G_j$ on $\hat{K}$ to element $K_i$,
that
\[
g_{i,j} = \frac{1}{2}(x_{i}+x_{i-1} + h_{i} G_j ), \quad j\in\mathbb{Z}_{k}, \,\,\,\mathrm{and}\,\,\, g_{N,k+1}=1.
\]
Based on the dual points $g_{i,j}$, we construct a dual partition
\[
\mathcal{T}_{h}^{*} = \{K_{1,0}^{*}\} \cup \{K_{i,j}^{*}: (i,j)\in\mathbb{Z}_{N}\times\mathbb{Z}_{k}\},
\]
where
\begin{align*}
&K_{1,0}^{*}=[0,g_{1,1}],\quad \mathrm{and} \quad K_{i,j}^{*} =[g_{i,j},g_{i,j+1}],\\
&g_{i,k+1}= g_{i+1,1},\,\, \forall\,  i\in \mathbb{Z}_{N-1}.
\end{align*}
The corresponding test function space $V_{h}$ is taken from the piecewise constant function space over $\mathcal{T}_{h}^{*}$,
which vanishes on the intervals $K_{1,0}^{*}\cup K_{N,k}^{*}$. Let
\[
V_{h} := \{v_h : \,v_h=\sum_{i=1}^{N}\sum_{j=1}^{k} v_{i,j}\psi_{i,j},\,\, (i,j)\in \mathbb{Z}_{N}\times\mathbb{Z}_{k}, \,\, v_{1,0}=v_{N,k}=0. \},
\]
where $v_{i,j}$ and $\psi_{i,j}=\chi[g_{i,j},\,g_{i,j+1}]$ are the constant and the characteristic function on $K_{i,j}^{*}$, respectively. Here, we have $\dim V_{h} = Nk-1 = \dim U_{h}^{k}$.

\subsection{FVE scheme}
Integrating (\ref{eq:BVP1D}) on each control volume $K_{i,j}^{*}=[g_{i,j},\,g_{i,j+1}]\in \mathcal{T}_{h}^{*}$, with integration by parts, we have
\begin{align} \label{eq:BVP_int}
p(g_{i,j})\,u'(g_{i,j}) - p(g_{i,j+1})\,u'(g_{i,j+1})
       + \int\limits_{ g_{i,j} }^{ g_{i,j+1} } q(x) u'(x) + r(x) u(x) \, \ud x
=  \int\limits_{ g_{i,j} }^{ g_{i,j+1} } f(x) \ud x.
\end{align}
For any $v_{h}\in V_{h}$, multiplying (\ref{eq:BVP_int}) with $v_{ij}$ and summing up over $K_{i,j}^{*}\in\mathcal{T}_{h}^{*}$, we have
\begin{align*}
\begin{array}{l}
  \sum\limits_{i=1}^{N} \sum\limits_{j=1}^{k} v_{i,j}\left(p(g_{i,j})\,u'(g_{i,j}) - p(g_{i,j+1})\,u'(g_{i,j+1})
       + \int\limits_{ g_{i,j} }^{ g_{i,j+1} } q(x) u'(x) + r(x) u(x) \, \ud x \right)
  =  \int\limits_{ 0 }^{ 1 } f(x)  v_h \, \ud x,
\end{array}
\end{align*}
which can also be written as
\[
\sum\limits_{i=1}^{N} \sum\limits_{j=1}^{k} [v_{i,j}] p(g_{i,j})\,u'(g_{i,j})
       + \sum\limits_{i=1}^{N} \sum\limits_{j=1}^{k} v_{i,j} \int\limits_{ g_{i,j} }^{ g_{i,j+1} } q(x) u'(x) + r(x) u(x) \, \ud x
=  \int\limits_{ 0 }^{ 1 } f(x)  v_h \ud x.
\]
Here $[v_{i,j}]=v_{i,j}-v_{i,j-1}$ is the jump of $v_{h}$ at point $g_{i,j}$, and $v_{i,0}=v_{i-1,k}$, $2\leq i\leq N$.

The finite volume element scheme for solving (\ref{eq:BVP1D}) is to find $u_{h} \in U_{h}^{k}$ such that
\begin{align} \label{eq:FVEM_BVP1D}
a_{h}(u_{h},v_{h})=(f,v_{h}), \qquad  \forall v_{h} \in V_{h},
\end{align}
where
\begin{align*}
a_{h}(u_h,v_{h}) =\sum\limits_{i=1}^{N} \sum\limits_{j=1}^{k} [v_{i,j}] p(g_{i,j})\,u_h'(g_{i,j})
       + \sum\limits_{i=1}^{N} \sum\limits_{j=1}^{k} v_{i,j} \int\limits_{ g_{i,j} }^{ g_{i,j+1} } q(x) u_h'(x) + r(x) u_h(x)  \ud x,
\end{align*}
which can also be written as
\begin{align*}
a_{h}(u_h,v_{h}) =\sum\limits_{ K_{i}\in\mathcal{T}_{h} } \sum\limits_{ K_{i,j}^{*} \in \mathcal{T}_{h}^{*} }
-( p \, u_h' \, v_{h})|_{\partial K_{i,j}^{*} \cap K_{i}}      + \sum\limits_{K_{i}\in\mathcal{T}_{h}} \int\limits_{ K_{i} }  (q(x) u_h'(x) + r(x) u_h(x))v_{h}  \ud x.
\end{align*}

\subsection{Notations about interpolation operators}
\label{subsec:notations}
For convenience of the proof, we define the following two operators $\Pi_{h}^{k}$ and $\Pi_{h}^{k,*}$.

$\Pi_{h}^{k}: \mathit{H}^{1}(\Omega)\rightarrow \mathit{U}_{h}^{k}$, the piecewise $k$-order Lagrange interpolation operator.


\begin{figure}[!ht]
    \centering
    \includegraphics[width=400pt]{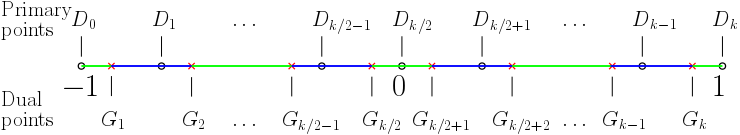}
    \caption{Computing nodes and dual points on $\hat{K}=[-1,1]$ for even order FVE schemes} \label{fig:nodes_even}
\end{figure}

$\Pi_{h}^{k,*}: \mathit{U}_{h}^{k} \rightarrow \mathit{V}_{h}$, a piecewise constant projection operator based on the dual mesh $\mathcal{T}_{h}^{*}$.
Let $-1= D_0 <\dots<D_k=1$ be $k+1$ symmetric points, to define the interpolation nodes of $\Pi_{h}^{k,*}$, on $\hat{K}=[-1,1]$. There should be $l-1$ parameters $a_{i}$ to locate all $D_{j}$ with constraints
    \begin{align}\label{eq:a_s}
    1>a_{1}>\cdots>a_{l-1} >0,  \quad
     k=2l-1\,(\mathrm{or}\, k=2l), \quad l=1,2,3,\dots ,
    \end{align}
     such that
    \begin{align} \label{eq:Ds}
    \left\{
      \begin{array}{cl}
        - D_{j} = D_{k-j} = a_j, & \mathrm{for}\,\,1\leq j \leq l-1, \\
        D_{l} =0, & \mathrm{if}\,\,\,k=2l. \\
      \end{array}
    \right.
    \end{align}
    The interpolation nodes of $\Pi_{h}^{k,*}$ are defined by the affine transformations of the points $D_j$ on $\hat{K}$ to the elements $K_i$, that
    \[
    d_{1,0}=-1,\,\,\,\mathrm{and}\,\,\,d_{i,j} = \frac{1}{2}(x_{i}+x_{i-1} + h_{i} D_j ), \quad j\in\mathbb{Z}_{k}.
    \]
    Then, for any $w_{h}\in\mathit{U}_{h}^{k}$, $\Pi_{h}^{k,*} w_{h}$ is given by
    \begin{align*}
      \Pi_{h}^{k,*}w_{h}\vert_{K_{i,j}^{*}}=w_{h}(d_{i,j}),\quad \forall K_{i,j}^{*}\in \mathcal{T}_{h}^{*}.
    \end{align*}

\begin{figure}[!ht]
    \centering
    \includegraphics[width=400pt]{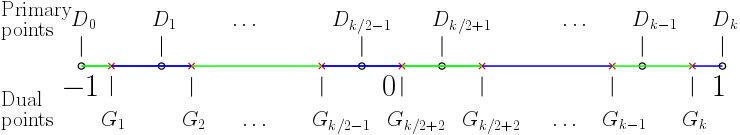}
    \caption{Computing nodes and dual points on $\hat{K}=[-1,1]$ for odd order FVE schemes} \label{fig:nodes_odd}
\end{figure}

\begin{rem}
$\Pi_{h}^{k,*}$ is only used for the analysis, and has no influence in practical computations.
\end{rem}


\section{The orthogonal condition and modified M-decomposition}
\label{sec:OrthCond}
The orthogonal condition and the modified M-decomposition (MMD) are two important tools in the superconvergence analysis of the FVEM.
The orthogonal condition can be used not only to constructing of the FVE schemes, but also to eliminate low-order terms in the analysis.
While, the MMD helps to find the proper superclose function $u_I$, which bridges the exact solution $u$ and the numerical solution $u_h$ at superconvergent points.

\subsection{The orthogonal condition}
For $k$-order FVEM, the \textbf{$k$-$r$-order orthogonal condition} ($k-1\leq r\leq 2(k-1)$) is the restriction on the dual meshes and interpolation nodes of the operator $\Pi_{h}^{k,*}$ for $r$-order orthogonality. Comparing with the $k$-$(k-1)$-order orthogonal condition proposed in \cite{Wang.2016}, the $k$-$r$-order orthogonal condition does not require the interpolation nodes of $\Pi_{h}^{k,*}$ to be uniform, to obtain higher-order orthogonality.

\begin{definition}[The $k$-$r$-order orthogonal condition]
A $k$-order FVE scheme or the corresponding dual strategy is called to satisfy the $k$-$r$-order orthogonal condition, if there exists a mapping $\Pi_{h}^{k,*}$ such that the following equations hold.
\begin{align}\label{eq:orth_condition1D}
&\int_{K}g(x)(w-\Pi_{h}^{k,*}w)(x) \ud x =0,\quad \forall g\in P^{r}(K),\,\, \forall w\in P^{1}(K),\,\, \forall K\in\mathcal{T}_{h}.
\end{align}
\end{definition}

It is enough to analyse the $k$-$r$-order orthogonal condition on $\hat{K}:=[-1,1]$.

\begin{lem} \label{lem:orth}
For any $k\geq2$, the $k$-$r$-order orthogonal condition (\ref{eq:orth_condition1D}) is equivalent to the following equations.
\begin{align}
   (1- \alpha_{1}^{\,\,i+1}) +
    \sum_{j=1}^{l-2}\,a_{j} ( \alpha_{j}^{\,\,i+1} - \alpha_{j+1}^{\,\,i+1} )\, +\,a_{l-1} \,\alpha_{l-1}^{\,\,i+1}
    &= \frac{i+1}{i+2},\quad \mathrm{for} \,\, k=2l-1,  \label{eq:orth_eqs_odd1}\\
   (1 - \alpha_{1}^{\,\,i+1} ) +
    \sum_{j=1}^{l-1}\,a_{j} ( \alpha_{j}^{\,\,i+1} - \alpha_{j+1}^{\,\,i+1} )
     &=  \frac{i+1}{i+2}, \quad
    \mathrm{for}\,\,  k=2l,  \label{eq:orth_eqs_even1}
\end{align}
where $i=1,3,5,\dots (i\leq r)$. Here, $\alpha_{j}$ $(j\in\mathbb{Z}_{j_{0}})$ are defined by (\ref{eq:alpha_s}),
which locate the dual points on $\hat{K}$. The parameters $a_{j}$ $(j\in\mathbb{Z}_{l-1})$ are defined by (\ref{eq:a_s}),
which confirm the interpolation nodes of $\Pi_{h}^{k,*}$.
\end{lem}

\begin{proof}
When $w$ is a constant on $\hat{K}$, $w\equiv \Pi_{h}^{k,*}w$. So (\ref{eq:orth_condition1D}) holds for constant $w$.
Noting that $P^{r}([-1,1])={\mathrm{ Span}}\{1,x,\dots,x^r\}$ and $P^{1}([-1,1])={\mathrm{Span}}\{1,x\}$, (\ref{eq:orth_condition1D}) is equivalent to
\begin{align}
\int_{-1}^{1} g (x-\Pi_{h}^{k,*}x) \ud x =0,\quad \forall g\in \{1,x,\dots,x^r\}.  \nonumber
\end{align}
Since $(x-\Pi_{h}^{k,*}x)$ is an odd function on $[-1,1]$, we arrive at
\begin{align} \label{eq:orth_eq_equivlent}
\int_{-1}^{1} x^{i} (x-\Pi_{h}^{k,*}x) \ud x =0,\quad {\mathrm{for}}\,\,i=1,3,5,\dots,\, i\leq r,
\end{align}
which implies
\begin{align} \label{eq:orth_eq_equiv_equations}
2\int_{0}^{1} x^{i} \Pi_{h}^{k,*}x \ud x = \frac{2}{i+2},\quad {\mathrm{for}}\,\, i=1,3,5,\dots,\, i\leq r.
\end{align}
Substituting the expression of $\Pi_{h}^{k,*}$ into (\ref{eq:orth_eq_equiv_equations}), one has
\begin{align*}
  \frac{1}{i+1}  \left((1-G_{k}^{\,\,i+1}) +\sum_{j=l}^{k-1} D_{j} ( G_{j+1}^{\,\,i+1} - G_{j}^{\,\,i+1} )   \right)
    &= \frac{1}{i+2}, \quad {\mathrm{for}}\,\, k=2l-1,\\
  \frac{1}{i+1}  \left((1-G_{k}^{\,\,i+1}) +\sum_{j=l+1}^{k-1} D_{j} ( G_{j+1}^{\,\,i+1} - G_{j}^{\,\,i+1} )   \right)
    &= \frac{1}{i+2}, \quad {\mathrm{ for}}\,\, k=2l,
\end{align*}
where $i=1,3,5,\dots$, and $i\leq r$.

Together with (\ref{eq:Gs}) and (\ref{eq:Ds}), we have Lemma~\ref{lem:orth}.
\end{proof}


\begin{lem} \label{lem:Integral}
Given $k$, for all $r\in\{t\in\mathbb{Z}:\,k-2\leq t\leq 2(k-1)\}$, there exists an operator $\Pi_{h}^{k,*}$, such that the corresponding FVE scheme satisfies the $k$-$r$-order orthogonal condition.
\end{lem}
\begin{proof}
From (\ref{eq:orth_eq_equivlent}), the $k$-$r$-order orthogonal condition is equivalent to
\begin{align} \label{eq:Integral}
   \sum\limits_{j=1}^{k} ( D_{j}-D_{j-1} ) G_{j}^{\,\,i+1} = \int_{-1}^{1}x^{i+1} \ud x, \quad {\mathrm{for}} \,\, i=1,3,\dots,\,\, i\leq r,
\end{align}
where $D_{j}$ are defined by (\ref{eq:Ds}) and $G_{j}$ are defined by (\ref{eq:Gs}). Summing the coefficients of $G_{j}^{\,\,i+1}$, one get
\[
\sum\limits_{j=1}^{k} ( D_{j}-D_{j-1} )=D_{k}-D_{0}=2.
\]
Thus, $( D_{j}-D_{j-1} )$s in $(\ref{eq:Integral})$ could be a selection as the weights of a $k$ points quadrature on $\hat{K}$.

An appropriate selection of $G_{j}$ and $D_{j}$ can make a $k$-points integration rule being accurate for $\tilde{k}$-order polynomial space
(see \cite{Davis.1984}), where $\tilde{k}\in[k-1, \,2k-1]$.
In other words, $r+1$ could be any integer in $[k-1,\,2k-1]$ such that there exists at least one $\Pi_{h}^{k,*}$
satisfying the $k$-$r$-order orthogonal condition. Thus we complete the proof.
\end{proof}


\begin{rem}\label{rem:orth}
It follows from (\ref{eq:orth_eq_equivlent}) that, for odd $r$, the $k$-$r$-order orthogonal condition is equivalent to the $k$-$(r+1)$-order orthogonal condition.
Thus, for all odd $k$-order FVE schemes with symmetric dual meshes, the $k$-$(k-1)$-order orthogonal condition is always satisfied.
While for the even $k$-order ($k=2l$) FVEM, if the $k$-$(k-1)$-order orthogonal condition is satisfied, the $k$-$k$-order orthogonal condition holds naturally.
\end{rem}

\subsection{The modified M-decomposition}
A proper {\em superclose function} $u_I$, which bridges the exact solution $u$ and the numerical solution $u_h$ at superconvergent points, is very important in the superconvergence analysis of the FEM/FVEM.
The M-decomposition technique (see \cite{Chen.1995,Cockburn.2017}) works well in constructing the appropriate superclose functions for the FEM, however it usually fails to be applied to the FVEM directly. For this reason, we propose the modified M-decomposition (MMD) technique to obtain an appropriate superclose function $u_I$ for the FVEM.

The M-functions on $\hat{K}=[-1,1]$, obtained by the integral of Legendre polynomials, are given by
\[
\hat{M}_{0}=1,\,\, \hat{M}_{1}=\xi,  \,\, \hat{M}_{2}=\frac{1}{2}(\xi^2-1), \,\, \dots,\,\, \hat{M}_{i+1}=\frac{1}{2^i\,i}\frac{\ud^{i-1}}{\ud \xi^{i-1}}(\xi^2-1)^{i},\,\,\dots,
\]
with properties
\begin{align*}
\left\{
\begin{array}{ll}
\hat{M}_{i}(\pm 1)=0,       & i=2,3,\dots,\\
(\hat{M}_{i},\hat{M}_j)=0,  & i\not= j\pm2.
\end{array}
\right.
\end{align*}

Suppose that $u\in H_{0}^{1}\cap H^{k+2}(\Omega)$ is the solution to (\ref{eq:BVP1D}), $u_{I}\in U_{h}^{k}$ to be determined later is a piecewise $k$-order approximation of $u$,
and $w_{I}\in U_{h}^{k}$ is an arbitrary piecewise $k$-order polynomial. Decompose $u$, $u_I$ and $w_{I}$ on an element $K\in \mathcal{T}_{h}$
with M-polynomials (\cite{Chen.1995,Cockburn.2017}).
\begin{align}
  u   =& \sum_{i=0}^{k} b_{i,K}^{u} M_{i} +b_{k+1}^{u} M_{k+1} + \sum_{i=k+2}^{\infty} b_{i,K}^{u} M_{i},
        \label{eq:u_decompose}\\
  u_I =& \sum_{i=0}^{k} b_{i,K}^{I} M_{i},  \label{eq:uI_decompose}\\
  w_I =&  \sum_{i=0}^{k} b_{i,K}^{w} M_{i}. \label{eq:wI_decompose}
\end{align}
Here, $M_{i}(x) = \hat{M}_{i}(\xi)$ ($x=x(\xi)$) is the $i$-order M-polynomial defined on $K$, and $\hat{M}_{i}(\xi)$ is the $i$-order M-polynomial on the reference element $\hat{K}$.
$b_{i,K}^{u}$ can be determined by $u$, and $b_{i,K}^{u}=O(h^{i})$ $(i\leq k+2)$ (\cite{Chen.1995,Cockburn.2017}).
Hereinafter, we omit the subscript $K$ without causing confusion.

An appropriate superclose function $u_I\in U_{h}^{k}$ should satisfy the following properties:
\begin{itemize}
  \item For superconvergence of the derivative, 1) $u_I'$ is superclose to $u'$ at the derivative superconvergent points, with order $O(h^{k+1})$; 2) $\|u_h-u_I\|_{1}=O(h^{k+1})$;

  \item For superconvergence of the function value, 1) $u_I$ is superclose to $u$ at the derivative superconvergent points, with order $O(h^{k+2})$; 2) $\|u_h-u_I\|_{0}=O(h^{k+2})$.
\end{itemize}
Generally speaking, these two desired superclose functions are consistent and could be a same function.
\begin{definition}[The modified M-decomposition (MMD) constrains]
The modified M-decomposition constraints on $K$ are given by $(\ref{eq:MMD_odd1})$-$(\ref{eq:MMD_odd2})$
and $(\ref{eq:MMD_even1})$-$(\ref{eq:MMD_even2})$.

For odd $k$-order ($k=2l-1$) FVE schemes,
    \begin{subequations}
    \begin{eqnarray}
      &b_{i}^{I} = b_{i}^{u},&  i\in\{0,1\}\cup\{3,5,\dots,k\},          \label{eq:MMD_odd1}    \\
      &\sum\limits_{t=1}^{l-1} b_{2t}^{*} \hat{M}_{2t}'(G_{m}) +b_{k+1}^{u} \hat{M}_{k+1}'(G_{m}) = 0, &   m\in\{1,2,\dots,l-1\},                    \label{eq:MMD_odd2}
    \end{eqnarray}
    \end{subequations}

For even $k$-order ($k=2l$) FVE schemes,
    \begin{subequations}
    \begin{eqnarray}
      &b_{i}^{I} = b_{i}^{u},&  i\in\{0,1\}\cup\{2,4,\dots, k\},          \label{eq:MMD_even1}    \\
      &\sum\limits_{t=2}^{l}b_{2t-1}^{*} \hat{M}_{2t-1}'(G_{m}) +b_{k+1}^{u} \hat{M}_{k+1}'(G_{m}) = 0, &   m\in\{1,2,\dots,l-1\}.                    \label{eq:MMD_even2}
    \end{eqnarray}
    \end{subequations}
Here, $b_{j}^{*} = b_{j}^{u}-b_{j}^{I}$ $(j=1,2,\dots)$.
\end{definition}

\begin{rem}
We can obtain the coefficients $b_{i}^{I}$ of $u_I$, by decomposing $u$ for $b_{i}^{u}$ first, and then solving the MMD constraints (\ref{eq:MMD_odd1})-(\ref{eq:MMD_odd2})
or (\ref{eq:MMD_even1})-(\ref{eq:MMD_even2}).
However, we only care about the properties of $u_{I}$ instead of $u_{I}$ itself in the proof.
So, we don't have to get the exact values of $b_{i}^{I}$ in practice.
\end{rem}

\begin{lem} \label{lem:MMD}
Let $u\in H_{0}^{1}\cap H^{k+2}(\Omega)$. If $u_I$ satisfies the \textbf{MMD} constraints,
the error of $u_I$ approximating $u$ can be estimated by
\begin{align}
\|u-u_I\|_{1} \lesssim h^{k}\|u\|_{k+1},    \label{eq:u_uI_H1}        \\
\|u-u_I\|_{0} \lesssim h^{k+1}\|u\|_{k+1}.  \label{eq:u_uI_L2}
\end{align}
\end{lem}
\begin{proof}
From (\ref{eq:MMD_odd1}) and (\ref{eq:MMD_even1}), one can write the difference between $u$ and $u_I$ on $K$ as
\begin{align}
  u - u_I  = R_{K}+r_{K},   \label{eq:u_uI_decomposition}
\end{align}
where
\begin{align*}
  R_{K} =& \left\{ \begin{array}{ll}
  \sum\limits_{t=1}^{l-1} b_{2t}^{*} M_{2t} +b_{k+1}^{u} M_{k+1} ,  &   k=2l-1,\\
  \sum\limits_{t=2}^{l} b_{2t-1}^{*} M_{2t-1} +b_{k+1}^{u} M_{k+1} ,  &   k=2l,\\
\end{array}
  \right.\\
  r_{K} =& \sum\limits_{i=k+2}^{\infty} b_{i}^{u} M_{i},  \qquad\qquad   k=2l-1\,\, (\mathrm{or}\,\,  k=2l).
\end{align*}
It's easy to verify that
\begin{align}
r_{K}=O(h^{k+2}),\quad \mathrm{and}\quad r_{K}'=(h^{k+1}).
\end{align}
Thus, the remaining work is to prove $R_{K}=O(h^{k+1})$ and $R_{K}'=O(h^{k})$.

For odd $k=2l-1$, we rewrite (\ref{eq:MMD_odd2}) into the matrix form
\begin{align*}
\mathbf{B} \, \mathbf{b}^{*} = b_{k+1}^{u} \, \mathbf{f}_{M},
\end{align*}
where
\begin{align*}
\mathbf{B} =&
\left(
  \begin{array}{cccc}
    \hat{M}_{2}'(G_{1}) & \hat{M}_{4}'(G_{1}) & \dots & \hat{M}_{2(l-1)}'(G_{1}) \\
    \hat{M}_{2}'(G_{2}) & \hat{M}_{4}'(G_{2}) & \dots & \hat{M}_{2(l-1)}'(G_{2}) \\
    \dots & \dots & \dots & \dots \\
    \hat{M}_{2}'(G_{l-1}) & \hat{M}_{4}'(G_{l-1}) & \dots & \hat{M}_{2(l-1)}'(G_{l-1}) \\
  \end{array}
\right)_{(l-1)\times (l-1)}, \\
\mathbf{b}^{*} =& (\,b_{2,K}^{*},\, b_{4,K}^{*}, \dots,\, b_{2(l-1),K}^{*}\,)^{T}_{1\times(l-1)}, \\
\mathbf{f}_{M} =& -(\, \hat{M}_{2(l+1)}'(G_{1}),\,  \hat{M}_{2(l+1)}'(G_{2}), \dots,\, \hat{M}_{2(l+1)}'(G_{l-1}) \,)^{T}_{1\times(l-1)}.
\end{align*}
Here, $\hat{M}_{2i}$ ($i=1,2,\dots,l-1$) and $\hat{M}_{k+1}$ are linearly independent M-functions on the reference element $\hat{K}$.
From the properties of M-functions, we can conclude that, $\mathbf{B}$ is invertible and the elements of $\mathbf{B}^{-1}$ and $\mathbf{f}_{M}$ are $O(1)$,
which are independent on $h$ and $K$.
%
%
Thus we have
\[
 \mathbf{b}^{*} = b_{k+1}^{u} \, \mathbf{B}^{-1}\, \mathbf{f}_{M}.
\]
Further more, since $b_{k+1}^{u}=O(h^{k+1})$, one can get
\[
b_{2t}^{*} = b_{k+1}^{u} \, O(1) = O(h^{k+1}),\quad t=1,\dots,(l-1),
\]
which leads to
\begin{align} \label{eq:RK_estimate}
R_{K}=O(h^{k+1}),\quad \mathrm{and}\quad R_{K}'=O(h^{k}).
\end{align}

Similar results to (\ref{eq:RK_estimate}) can be obtained for even $k$.
\end{proof}

Considering the symmetry of the M-functions on $\hat{K}$, (\ref{eq:MMD_odd2}) and (\ref{eq:MMD_even2}) give
\begin{align}
\sum\limits_{t=1}^{l-1} b_{2t}^{*} \hat{M}_{2t}'(G_{m}) +b_{k+1}^{u} \hat{M}_{k+1}'(G_{m}) =& 0, \quad   m\in\{1,2,\dots,k\}, \quad \mathrm{for}\,\,k=2l-1,          \label{eq:MMD_odd2_full}            \\
\sum\limits_{t=2}^{l} b_{2t-1}^{*} \hat{M}_{2t-1}'(G_{m}) +b_{k+1}^{u} \hat{M}_{k+1}'(G_{m}) =& 0, \quad   m\in\{1,2,\dots,k\}, \quad \mathrm{for}\,\,k=2l.          \label{eq:MMD_even2_full}
\end{align}
It follows from the linear affine mapping from $\hat{K}$ to $K$ and (\ref{eq:u_uI_decomposition}) that
\[
R_{K}'(g_{i,j}) = 0.
\]
Denote the $(k+1)$ roots of $R_{K}(x)$ by $\mathbb{P}_{0}=\{z_{K,j}^{0}: j=0,\dots,k \}$, and the $k$ roots of $R_K'(x)$ by $\mathbb{P}_{1}=\{g_{i,j}: j=1,\dots,k, \,\, K=K_i \}$.
Then one has $r_{K}'=O(h^{k+1})$ and
\[
u'(g_{i,j})-u_I'(g_{i,j}) = R_{K}'(g_{i,j}) + r_{K}'(g_{i,j}) = r_{K}'(g_{i,j}) = O(h^{k+1}).
\]

\begin{lem} \label{lem:superclose}
Under the same assumptions to Lemma~$\ref{lem:MMD}$, $u_I'$ is superclose to $u'$ on $\mathbb{P}_{1}$, and $u_I$ is superclose to $u$ on $\mathbb{P}_{0}$. That is
\begin{align}
\max_{z\in\mathbb{P}_{1}} |u'(z)-u_I'(z)| \lesssim h^{k+1}\|u\|_{k+2},    \label{eq:u_uI_derivative}    \\
\max_{z\in\mathbb{P}_{0}} |u(z)-u_I(z)| \lesssim h^{k+2}\|u\|_{k+2}.      \label{eq:u_uI_functionvalue}
\end{align}
\end{lem}

\begin{rem} \label{rem:correction}
The remaining thing is that the above $u_I$ defined element by element might not be continuous on $\Omega$. However, from lemma~\ref{lem:superclose}, $u-u_I=O(h^{k+2})$ on the endpoints of each $K\in\mathcal{T}_h$. Thus, we can simply use a high order correction $r_{\mathrm{corr}}=O(h^{k+2})$, which does not necessarily to be continuous on $\Omega$, to obtain a continuous $\tilde{u}_I$, such that $u-\tilde{u}_I=0$ on the endpoints of $K$ and
\[
\tilde{u}_{I} = u_I + r_{\mathrm{corr}},
\]
which inherits the superclose properties of $u_I$. And, on $K\in\mathcal{T}_{h}$,
\[
u-\tilde{u}_{I} = R_{K}+\tilde{r}_K,
\]
where, $\tilde{r}_{K}=r_{K}-r_{\mathrm{corr}}$. And, $\tilde{r}_{K}$ inherits the properties $\tilde{r}_{K}'=O(h^{k+1})$ and $\tilde{r}_{K}=O(h^{k+2})$ from $r_K$.

In the following analysis, we'll still write $\tilde{u}_I$ and $\tilde{r}_{K}$ as $u_I$ and $r_{K}$ without causing confusion.

\end{rem}

\section{Superconvergence}
\label{sec:Superconvergence}

Lemma~\ref{lem:superclose} presents the supercloseness between the exact solution $u$ and its approximation $u_I$. In this section, we prove the global superconvergence properties that $\|u_h-u_I\|_{1}=O(h^{k+1})$ and $\|u_h-u_I\|_{0}=O(h^{k+2})$. Then, the natural superconvergence results follow natually, that the numerical solution $u_h'$ superconverges to $u'$ on the derivative superclose points in $\mathbb{P}_{1}$ with $(k+1)$-order, and $u_h$ superconverges to $u$ on the function value superclose points in $\mathbb{P}_{0}$ with $(k+2)$-order.

\subsection{Superconvergence of the derivative}
\begin{thm}[Superconvergence of the derivative] \label{thm:super_H1}
Let $u\in H_{0}^{1}\cap H^{k+2}(\Omega)$ be the solution of $(\ref{eq:BVP1D})$,  and $\mathcal{T}_{h}$ is regular. For $k$-order Lagrange trial function space $U_{h}^{k}$, choose $\mathcal{T}_{h}^{*}$ satisfying the $k$-$(k-1)$-order orthogonal condition. For the $u_I\in U_{h}^k$ satisfying the MMD constraints $(\ref{eq:MMD_odd1})$-$(\ref{eq:MMD_odd2})$ or $(\ref{eq:MMD_even1})$-$(\ref{eq:MMD_even2})$, we have the weak estimate of the first type 
\begin{equation} \label{eq:Weakestimate_thm}
|a_{h}(u-u_I, \Pi_{h}^{k,*}w_{I})| \leq \mathit{C}h^{k+1} \|u\|_{k+2} \| w_{I} \|_{1},
\quad \forall\, w_{I}\in U_{h}^{k}.
\end{equation}
Consequently,
\begin{equation} \label{eq:Superconv1_thm}
\| u_{h}-u_{I}\|_{1}\leq \mathit{C}h^{k+1} \|u\|_{k+2}.
\end{equation}
\end{thm}
In order to prove theorem~\ref{thm:super_H1}, we first prove the following lemma.
\begin{lem} \label{lem:RK_Piw}
For the difference between $u$ and $u_I$ on $K$ $(\ref{eq:u_uI_decomposition})$, we have
\begin{align} \label{eq:RK_Piw}
  \int_{-1}^{1} R_{K}''\, \Pi_{h}^{k,*} (w_I-\Pi_h^1 w_I) \ud \xi = 0.
\end{align}
\end{lem}
\begin{proof}
For odd $k=2l-1$, from (\ref{eq:MMD_odd2_full}), we have
\begin{align*}
\sum\limits_{m=1}^{ k }
(\sum\limits_{t=1}^{l-1} b_{2t}^{*} \hat{M}_{2t}'(G_{m}) +b_{k+1}^{u} \hat{M}_{k+1}'(G_{m}))
\, \left( \hat{M}_{j}( D_{m-1} ) - \hat{M}_{j}( D_{m} ) \right) &= 0,  \quad j=2,4,\dots,k-1.
\end{align*}
Combining with the fact that 
\begin{align*}
  & \int_{-1}^{1} \hat{M}_{i}''\, \Pi_{h}^{k,*} \hat{M}_{j} \ud \xi \nonumber\\
  =& \hat{M}_{j}( D_{0} ) \int_{-1}^{G_{1}} \hat{M}_{i}'' dx
         +\sum\limits_{m=1}^{k-1} \hat{M}_{j}( D_{m} )
            \int_{ G_{m} }^{ G_{m+1} }
                \hat{M}_{i}'' dx
         + \hat{M}_{j}( D_{k} ) \int_{G_{k}}^{1} \hat{M}_{i}'' dx  \nonumber\\
  =& \hat{M}_{j}( D_{0} )\,\, \hat{M}_{i}'|_{-1}^{ G_{1} }
         +\sum\limits_{m=1}^{ k-1 } \left( \hat{M}_{j}( D_{m} ) \,\,  \hat{M}_{i}'|_{ G_{m} }^{ G_{m+1} } \right)
         + \hat{M}_{j}( D_{k} )\,\,  \hat{M}_{i}'|_{ G_{k} }^{1}  \nonumber  \\
  =& \sum\limits_{m=1}^{ k } \hat{M}_{i}'(G_{m}) \left( \hat{M}_{j}( D_{m-1} ) - \hat{M}_{j}( D_{m} ) \right),
           \nonumber
\end{align*}
we have
\begin{align} \label{eq:M_combine1}
  \int_{-1}^{1} (\sum\limits_{t=1}^{l-1} b_{2t}^{*} \hat{M}_{2t} +b_{k+1}^{u} \hat{M}_{k+1})''\, \Pi_{h}^{k,*} \hat{M}_{j} \ud \xi = 0, \quad   j= 2,4,\dots,k-1.
\end{align}

On the other hand, since $\Pi_{h}^{k,*} \hat{M}_{j}$ ($j=3,5,\dots,k$) are odd functions on $\hat{K}$  and $(\sum\limits_{t=1}^{l-1} b_{2t}^{*} \hat{M}_{2t} +b_{k+1}^{u} \hat{M}_{k+1})''$ is an even function, we have
\begin{align} \label{eq:M_combine2}
  \int_{-1}^{1} (\sum\limits_{t=1}^{l-1} b_{2t}^{*} \hat{M}_{2t} +b_{k+1}^{u} \hat{M}_{k+1})''\, \Pi_{h}^{k,*} \hat{M}_{j} \ud \xi = 0, \quad  j= 3,5,\dots,k.
\end{align}
It follows from (\ref{eq:M_combine1}) and (\ref{eq:M_combine2}) that
\begin{align*}
  \int_{-1}^{1} (\sum\limits_{t=1}^{l-1} b_{2t}^{*} \hat{M}_{2t} +b_{k+1}^{u} \hat{M}_{k+1})''\, \Pi_{h}^{k,*} (\sum\limits_{j=2}^{k} b_{j,K}^{w}\hat{M}_{j}) \ud \xi = 0.
\end{align*}
By (\ref{eq:wI_decompose}), (\ref{eq:u_uI_decomposition}),and the linear affine mapping from $K$ to $\hat{K}$, one has the conclusion (\ref{eq:RK_Piw}).

With a similar procedure, (\ref{eq:RK_Piw}) also holds for even $k=2l$.
\end{proof}

Now, we are ready to give the proof of Theorem~\ref{thm:super_H1}.
\begin{proof}
We first prove the weak estimate of the first type (\ref{eq:Weakestimate_thm}).
\begin{align}
  a_{h}(u-u_I, \Pi_{h}^{k,*}w_{I}) = E1+E2+E3,
\end{align}
where
\begin{align}
  E_1 =& a_{h}(u-u_I, \Pi_{h}^{k,*}(w_{I}- w_{I,1}) ), \nonumber \\
  E_2 =& a_{h}(u-u_I, \Pi_{h}^{k,*} w_{I,1}) -a(u-u_I, w_{I,1}),  \nonumber \\
  E_3 =& a(u-u_I, w_{I,1} ),  \nonumber
\end{align}
with $w_{I,1}=\Pi_{h}^{1}w_I$.

To estimate $E_1$, one can  use (\ref{eq:u_uI_decomposition}) and (\ref{eq:RK_Piw}) to get
\begin{align}  \label{eq:Superconv_E1_1}
 \left|\int_{K} (u-u_I)'' \, \Pi_{h}^{k,*}(w_{I}-w_{I,1})\, \ud x \right|
   =& \left|\int_{K} (R_{K}+r_{K})''\, \Pi_{h}^{k,*}(w_{I}-w_{I,1})\, \ud x \right|  \nonumber\\
   =& \left|\int_{K} (r_{K})'' \, \Pi_{h}^{k,*}(w_{I}-w_{I,1})\, \ud x \right|    \nonumber\\
   \lesssim & h^{k+1}\|u\|_{k+2,K} \|w_{I}\|_{1,K}.
\end{align}
Here, the hidden constant is independent of $h$. Noting that $w_I-w_{I,1}=0$ on the endpoints of each $K$, we have the estimate for the diffusion term, which is the first term of $a_{h}(u-u_I, \Pi_{h}^{k,*}(w_{I}-w_{I,1}) )$.
\begin{align}  \label{eq:Superconv_E1_2}
& \bigg{|}\sum\limits_{ K\in\mathcal{T}_{h} } \sum\limits_{ K^{*} \in \mathcal{T}_{h}^{*} }
-( p \, (u-u_{I})' \, \Pi_{h}^{k,*}(w_I - w_{I,1}))|_{\partial K^{*} \cap K}\bigg{|}  \nonumber \\
=& \bigg{|}\sum\limits_{K\in\mathcal{T}_{h}} ( p(u-u_{I})'\, \Pi_{h}^{k,*}(w_I - w_{I,1}) )|_{\partial K}
    - \sum\limits_{K\in\mathcal{T}_{h}} \int_{K}(p(u-u_{I})')'\,  \Pi_{h}^{k,*}(w_I - w_{I,1})\,\ud x\bigg{|}  \nonumber \\
\leq&  \bigg{|}\sum\limits_{K\in\mathcal{T}_{h}} \int_{K}((p-p_{0,K})(u-u_{I})')'\,  \Pi_{h}^{k,*}(w_I - w_{I,1})\,\ud x \bigg{|}  \nonumber\\
& +\bigg{|} \sum\limits_{K\in\mathcal{T}_{h}} \int_{K}p_{0,K}(u-u_{I})''\,  \Pi_{h}^{k,*}(w_I - w_{I,1})\,\ud x\bigg{|}   \nonumber\\
\lesssim& h^{k+1}\|u\|_{k+2} \|w_{I}\|_{1}.
\end{align}
It follows from (\ref{eq:u_uI_L2}) that
\begin{align} \label{eq:Superconv_E1_3}
      &  \left| \int_{ K} q (u-u_I)'\, \Pi_{h}^{k,*}(w_{I}-w_{I,1}) + r (u-u_I) \,    \Pi_{h}^{k,*}(w_{I}-w_{I,1}) \ud x  \right| \nonumber \\
      \lesssim& \|u-u_I\|_{1,K} \,\|w_{I}-w_{I,1}\|_{0} + \|u-u_I\|_{0,K} \,\|w_{I}-w_{I,1}\|_{0}   \nonumber\\
      \lesssim& h^{k+1} \|u\|_{k+1,K} \, \| w_{I} \|_{1,K},
\end{align}
where the hidden constant depends on $q$ and $r$. Then, (\ref{eq:Superconv_E1_2}) and (\ref{eq:Superconv_E1_3}) yield
\begin{align}  \label{eq:Superconv_E1}
  |E_1| \lesssim&  h^{k+1} \|u\|_{k+2,K} \, \| w_{I} \|_{1,K},
\end{align}
where the hidden constant is dependent on $p$, $q$, and $r$.

For $E_2$, it follows from the $k$-$(k-1)$-order orthogonal condition (\ref{eq:orth_condition1D}) and the inverse inequality that
\begin{align} \label{eq:Superconv_E2_0}
& \bigg{|} \sum\limits_{K\in\mathcal{T}_{h}}\int_{K}(p(u-u_{I})')'\,
    \big(\Pi_{h}^{1}w-\Pi_{h}^{k,*}(\Pi_{h}^{1}w)\big)\ud x  \bigg{|}    \nonumber \\
\leq& \bigg{|} \sum\limits_{K\in\mathcal{T}_{h}}\int_{K}((p-p_{0,K})(u-u_{I})')'\,
    \big(\Pi_{h}^{1}w-\Pi_{h}^{k,*}(\Pi_{h}^{1}w)\big)\ud x  \bigg{|}     \nonumber \\
& +\bigg{|} \sum\limits_{K\in\mathcal{T}_{h}} p_{0,K} \int_{K} \left((u-u_{I})'' - \Pi_{h}^{k-1}(u-u_{I})'' \right)\,
    \big(\Pi_{h}^{1}w-\Pi_{h}^{k,*}(\Pi_{h}^{1}w)\big)\ud x  \bigg{|}      \nonumber \\
\leq&  C h \|p\|_{2,\infty} \|u-u_{I}\|_{2} \, \|\Pi_{h}^{1}w-\Pi_{h}^{k,*}(\Pi_{h}^{1}w)\|_{0} \nonumber\\
& +\bigg{|} \sum\limits_{K\in\mathcal{T}_{h}}\int_{K} p_{0,K} \left( u'' - \Pi_{h}^{k-1}u'' \right)\,
    \big(\Pi_{h}^{1}w-\Pi_{h}^{k,*}(\Pi_{h}^{1}w)\big)\ud x   \bigg{|}      \nonumber\\
\leq&  C h^2 \|p\|_{2,\infty} \|u-u_{I}\|_{2} \, \|w\|_{1}  + C \| p\|_{0,\infty} h^{k+1} \| u \|_{k+2}\, \|w\|_{1}    \nonumber \\
\lesssim&  h^{k+1} \|u\|_{k+2} \, \|w\|_{1},    
\end{align}
where $p_{0,K}$ is the average of $p$ on $K$. By the integration by parts and the fact $u(x_i)-u_I(x_i)=0$, we can obtain
\begin{align}\label{eq:Superconv_E2}
  |E_2| = &  \bigg{|} \sum\limits_{K\in\mathcal{T}_{h}} \int_{K}(p(u-u_{I})')'\,  (\Pi_{h}^{k,*}w_{I,1} -w_{I,1} )\,\ud x\bigg{|}    \nonumber\\
  & + \bigg{|}\sum\limits_{K\in\mathcal{T}_{h}} \int\limits_{ K }  (q\, (u-u_{I})' + r\, (u-u_{I}))\, (\Pi_{h}^{k,*} w_1 -w_1) \ud x\bigg{|} \nonumber\\
  \lesssim &  h^{k+1} \|u\|_{k+2}  \|w_{I}\|_{1},
\end{align}
and
\begin{align} \label{eq:Superconv_E3}
  |E_3| =& \bigg{|}\sum\limits_{K\in\mathcal{T}_{h}}\int_{K} \big( p(u-u_{I})'\,
    w_{I,1}' + q (u-u_{I})' \,w_{I,1}+ r (u-u_{I})\,w_{I,1}\big) \ud x\bigg{|}   \nonumber\\
    =& \bigg{|}\sum\limits_{K\in\mathcal{T}_{h}}\int_{K} \big( (u-u_{I})\,
    p'\,w_{I,1}' + (u-u_{I}) \,(q\,w_{I,1})'+ r (u-u_{I})\,w_{I,1}\big) \ud x\bigg{|}   \nonumber\\
    \lesssim& \sum\limits_{K\in\mathcal{T}_{h}}  \|u-u_{I}\|_{0,K} \,\|w_{I,1}\|_{1,K}   \nonumber\\
    \lesssim&  h^{k+1} \|u\|_{k+1} \,\|w_I\|_{1}.
\end{align}

Combining (\ref{eq:Superconv_E1})-(\ref{eq:Superconv_E3}), one can get (\ref{eq:Weakestimate_thm}) immediately.
Further more, by using the inf-sup condition (\ref{eq:infsup}), we have the superconvergence of the derivative (\ref{eq:Superconv1_thm}).
\end{proof}

As a side product of the superconvergence of the derivative, we give the $L^2$ estimates without proof.
\begin{thm}[$L^2$ estimates] \label{thm:L2}
Let $u\in H_{0}^{1}\cap H^{k+2}(\Omega)$ be the solution of $(\ref{eq:BVP1D})$, and $\mathcal{T}_{h}$ be regular. For $k$-order FVEM, if the $k$-$(k-1)$-order orthogonal condition holds,
we have the following optimal $L^2$ estimates 
\begin{equation} \label{eq:L2_thm}
\| u-u_{h}\|_{0}\leq \mathit{C}h^{k+1} \|u\|_{k+2}.
\end{equation}
\end{thm}


\subsection{Superconvergence of the function value}


%

\begin{thm}[Superconvergence of the function value]  \label{thm:Super_fun}
Let $u\in H_{0}^{1}\cap H^{k+3}(I)$ be the solution to $(\ref{eq:BVP1D})$,  and $\mathcal{T}_{h}$ be regular.
For $k$-order Lagrange trial function space $U_{h}^{k}$, if a FVE scheme satisfies the $k$-$k$-order orthogonal condition and $u_I\in U_{h}^k$ satisfies
the MMD constraints $(\ref{eq:MMD_odd1})$-$(\ref{eq:MMD_odd2})$ or $(\ref{eq:MMD_even1})$-$(\ref{eq:MMD_even2})$, we have
\begin{equation} \label{eq:Superconv2_thm}
\| u_{h}-u_{I}\|_{0}\leq \mathit{C}h^{k+2} \|u\|_{k+3}.
\end{equation}
\end{thm}
\begin{proof}
We begin with the Aubin-Nistche technique. Introduce an auxiliary problem: For $\forall g\in L^{2}(\Omega)$, find $w \in H^{1}_{0}(\Omega)$ such that
\begin{equation} \label{eq:dualproblem}
a(v,w)=(g,v),\quad \forall v\in H^{1}_{0}(\Omega),
\end{equation}
where
\begin{align*}
a( v , w) = {\int_{\Omega} pv'w'+ (q v' + r v) w \ud x,}    \quad
(g,v) = \int_{\Omega} \, g \, v \,\ud x.
\end{align*}
Take $g=u_h-u_I$ to give
\begin{align}
  \|u_h - u_I\|^2 =& a(u_h - u_I, w) \nonumber\\
   =& a(u_h - u_I, w- w_{1} ) + a(u - u_I, w_{1})        \nonumber\\
   &   - \left( a(u - u_h, w_{1})   - a_{h}(u-u_{h}, \Pi_h^{k,*} w_{1})  \right)  \nonumber\\
   =& E_4 +E_5+E_6.
\end{align}
where $w_{1}=\Pi_{h}^{1}w$ and
\begin{align*}
  E_4 =& a(u_h - u_I, w-w_{1}), \\
  E_5 =& a(u - u_I, w_{1}), \\
  E_6 =& -\left( a(u - u_h, w_{1})   - a_{h}(u-u_{h}, \Pi_h^{k,*} w_{1})  \right) .
\end{align*}

It follows from (\ref{eq:Superconv1_thm}) that
\begin{align}
|E_4| \lesssim \|u_h - u_I\|_{1} \|w-w_{1}\|_{1} \lesssim h^{k+2} \|u\|_{k+2}\|w\|_{2}.  \label{eq:E4}
\end{align}

For $E_5$, using (\ref{eq:u_uI_decomposition}) and the quasi-orthogonality of the M-functions, we have $\int_{K} u-u_I \,\ud x=\int_{K} (R_{K}+r_{K})\,1 \,\ud x=0$.
Taking the correction term $r_{\mathrm{corr}}$ in (\ref{rem:correction}) into consideration, we get
\[
\int_{K} u-u_I \,\ud x=\int_{K} R_{K}+r_{K}-r_{\mathrm{corr}} \,\ud x = \int_{K} -r_{\mathrm{corr}} \,\ud x = O(h^{k+2})\|u\|_{k+2,K}.
\]
Then, noticing that $w_{1}'$ is a constant on $K$, a similar procedure of (\ref{eq:Superconv_E3}) gives
\begin{align} \label{eq:E5}
  |E_5| =& \left|\sum\limits_{K\in\mathcal{T}_{h}}\int_{K} (u-u_{I})\,
    p'\,w_{1}' + (u-u_{I}) \,(q\,w_{1})'+ r (u-u_{I})\,w_{1} \ud x \right|   \nonumber\\
    \leq&\sum\limits_{K\in\mathcal{T}_{h}}  \left| \int_{K}  (u-u_{I})\,
    (p'-\Pi_{h}^{0}p')\,w_{1}'  \ud x \right|   \nonumber\\
    & + \sum\limits_{K\in\mathcal{T}_{h}}   \left| \int_{K} (u-u_{I}) \,((q\,w_{1})'-\Pi_h^0(q\,w_{1})'\,)+ (u-u_{I})\,(r\,w_{1}-\Pi_h^0 (r\,w_{1})) \ud x \right|   \nonumber\\
    & + C\sum\limits_{K\in\mathcal{T}_{h}} h^{k+2} \, \|u\|_{k+2,K} \|w_{1}\|_{1,K}  \nonumber\\
    \lesssim&  h^{k+2} \|u\|_{k+2} \,\|w\|_{1},
\end{align}
where the hidden constant is dependent on $p$, $q$ and $r$.

Following a similar procedure of $(\ref{eq:Superconv_E2})$, we have
\begin{align}\label{eq:E6}
  |E_6| \lesssim h^{k+2}\, \|u\|_{k+3}\, \|w\|_{1}.
\end{align}
Then, combining $(\ref{eq:E4})$-$(\ref{eq:E6})$ completes the proof.
\end{proof}


%
\begin{rem}  \label{rem:Compare_orthogonal}
For even $k$, the $k$-$k$-order orthogonal condition and the $k$-$(k-1)$-order orthogonal condition deliver same restrictions on the dual mesh, which means Theorems \ref{thm:super_H1} and \ref{thm:Super_fun} share same conditions for even $k$. While, for odd $k$, the $k$-$(k-1)$-order orthogonal condition is always satisfied for FVE schemes with symmetric dual meshes.
And, the restriction on dual mesh in Theorem~\ref{thm:Super_fun} is stronger than that in Theorem~\ref{thm:super_H1}, because the $k$-$k$-order orthogonal condition is stronger restrictions than the $k$-$(k-1)$-order orthogonal condition, for odd $k$.
\end{rem}

\section{Construction of FVE schemes with superconvergence}
\label{sec:construction}
From Theorems \ref{thm:super_H1} and \ref{thm:Super_fun}, we can conclude the relationships between the orthogonal condition and the convergence properties shown in Table~\ref{tab:relations}.

Following, we first construct FVE schemes with the help of the orthogonal condition, and then present how to construct FVE schemes with natural superconvergence in easy ways.

\subsection{Constructing the FVE schemes with the orthogonal condition}

\textbf{For odd-order FVEM ($k=2l-1$)}, the superconvergence of the derivative holds naturally for odd-order FVE schemes with symmetric dual meshes. And, when the $k$-$k$-order orthogonal condition is satisfied, there holds the superconvergence of the function value.

For the linear FVEM, the $1$-$1$-order orthogonal condition can not be reached, and the superconvergence of the function value can not be reached.

For the cubic (3-order) FVEM, the $3$-$3$-order orthogonal condition (\ref{eq:orth_condition1D}) leads to unique reasonable solution $\alpha_1=\sqrt{3/5}$.

For the quintic (5-order) FVEM, the $5$-$5$-order orthogonal condition (\ref{eq:orth_condition1D}) is equivalent to the following three restrictions
\begin{align}
\left\{\begin{array}{l}
  (1- a_{1}) \alpha_{1}^{\,\,2} + (a_1 -a_2)\alpha_{2}^{\,\,2}   = \frac{1}{3}, \\
  (1- a_{1}) \alpha_{1}^{\,\,4} + (a_1 -a_2)\alpha_{2}^{\,\,4}   = \frac{1}{5}, \\
  (1- a_{1}) \alpha_{1}^{\,\,6} + (a_1 -a_2)\alpha_{2}^{\,\,6}   = \frac{1}{7},
    \label{eq:orth_eqs_5order}
  \end{array}\right.
\end{align}
Noticing $0<\alpha_{2}<\alpha_{1}<1$, we have
\begin{align} \label{eq:alpha_relations_5order}
\alpha_{2} = \sqrt{\frac{\alpha_{1}^{2}/5 - 1/7}{\alpha_{1}^{2}/3 - 1/5}}, \quad \mathrm{and} \quad \alpha_{1}\in(\sqrt{5/7},\,1).
\end{align}

\begin{figure}[!ht]
    \centering
    \subfigure{
    \begin{minipage}[t]{.43\textwidth}
      \centering
      \includegraphics*[width=200pt]{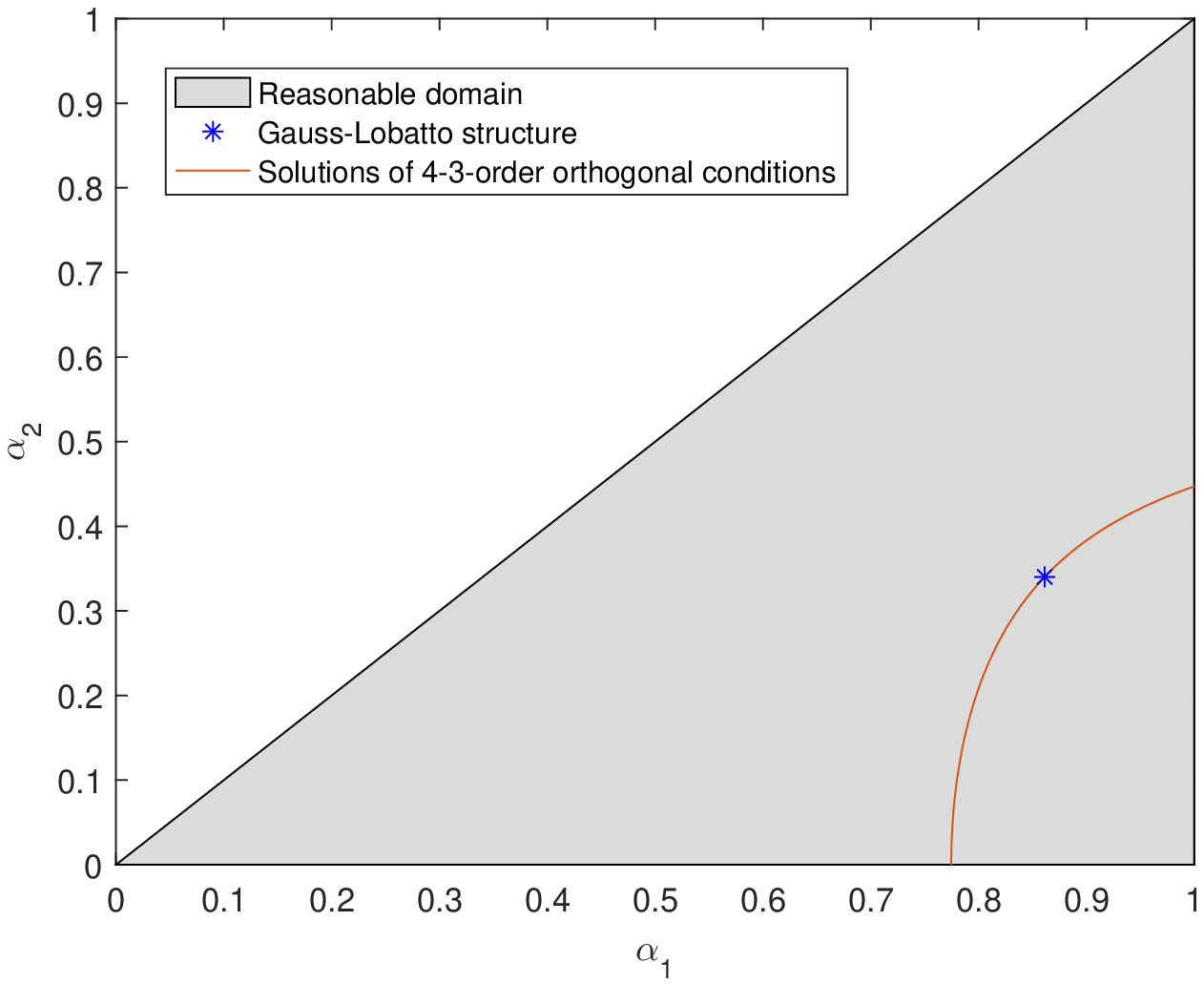}
      \title{(a)}
    \end{minipage}}
    \subfigure{
    \begin{minipage}[t]{.43\textwidth}
      \centering
      \includegraphics[width=200pt]{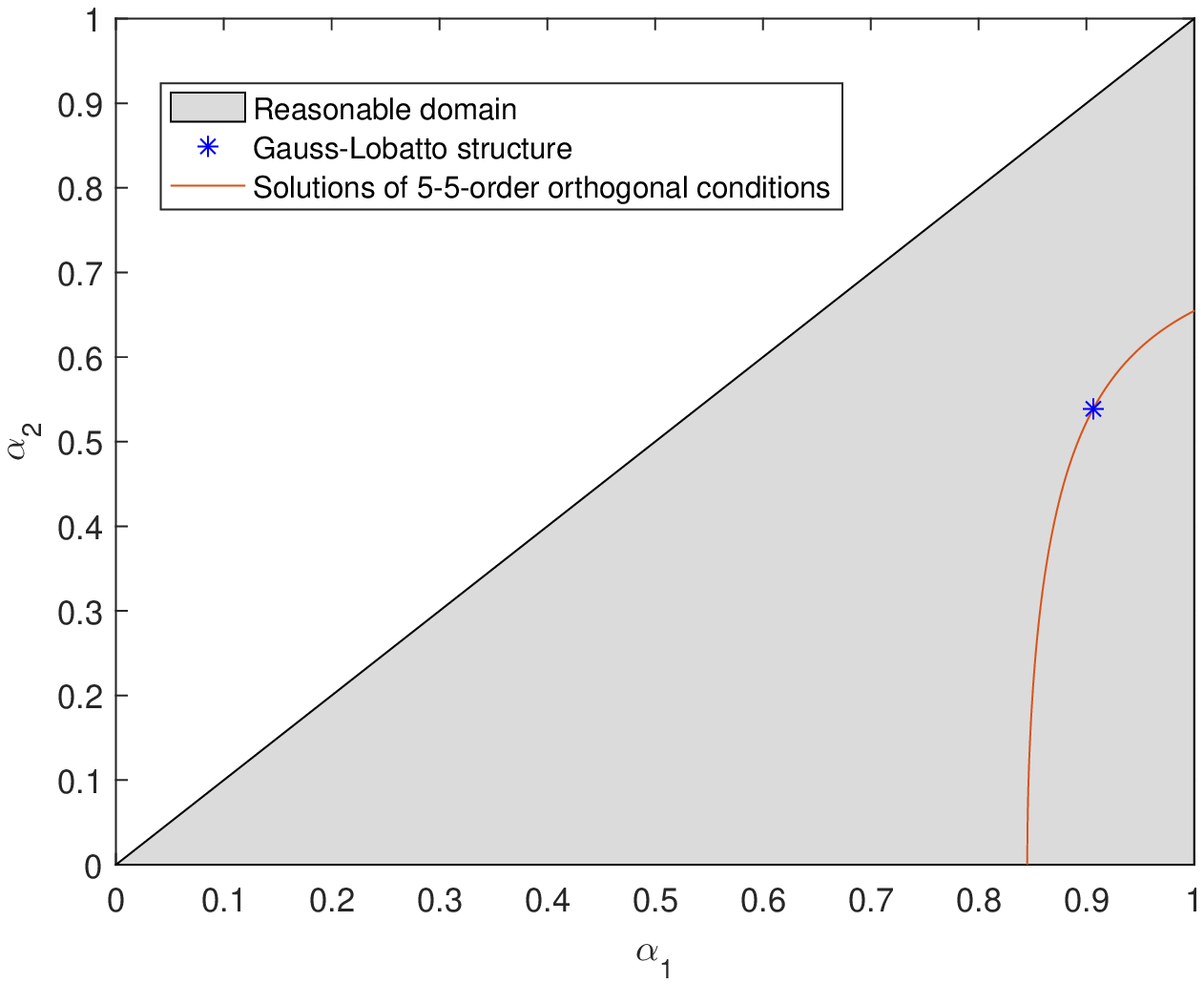}
      \title{(b)}
    \end{minipage}}
    \caption{The left figure shows the $4$-$3$-order orthogonal condition for $k=4$. And, the right figure shows the $5$-$5$-order orthogonal condition for $k=5$.}
    \label{fig:alpha4A5order}
\end{figure}

\textbf{For even-order FVEM ($k=2l$)}, when the $k$-$(k-1)$-order orthogonal condition is satisfied, there holds the superconvergence of the derivative, as well as the superconvergence of the function value.

For the quadratic (2-order) FVEM, the $2$-$1$-order orthogonal condition (\ref{eq:orth_condition1D}) has a unique reasonable solution $\alpha_{1}=\frac{\sqrt{3}}{3}$.

For the quartic (4-order) FVEM, the $4$-$3$-order orthogonal condition (\ref{eq:orth_condition1D}) can be derived from the following equations
\begin{align}
\left\{\begin{array}{l}
  1-\alpha_{1}^{\,\,2} + a_{1} ( \alpha_{1}^{\,\,2} - \alpha_{2}^{\,\,2} )  = \frac{2}{3}, \\
  1-\alpha_{1}^{\,\,4} + a_{1} ( \alpha_{1}^{\,\,4} - \alpha_{2}^{\,\,4} )  = \frac{4}{5}.
    \label{eq:orth_eqs_4order}
  \end{array}\right.
\end{align}
since $0<\alpha_{2}<a_1<\alpha_{1}<1$, we have
\begin{align}
\left\{\begin{array}{l}
\alpha_{1}= \sqrt{  \frac{1}{3} (1+ 2\, \sqrt{\frac{a_1}{5(1-a_{1})}})  }, \\
[-0.8ex]
\\
\alpha_{2}= \sqrt{  \frac{1}{3} (1- 2\, \sqrt{\frac{1-a_{1}}{5\, a_1}})  } ,
    \label{eq:4order_alpha}
  \end{array}\right.
\end{align}
where $4/9 \leq a_{1}< 5/6$.



\begin{rem}
For the FVEM, there are more than one scheme having the superconvergence properties for all $k\geq 3$. What's more, Fig.~\ref{fig:alpha4A5order} (a) for $4$-$3$-order (also $4$-$4$-order) orthogonal condition and Fig.~\ref{fig:alpha4A5order} (b) for $5$-$5$-order orthogonal condition show that, the Gauss-Lobatto structure is a spacial case of the orthogonal structure for FVEM.
\end{rem}

\subsection{Constructing the FVE schemes in easy ways}
\label{subsec:construction_easy}
For the convenience of use, we present the ways to freely choose the derivative superconvergent points (for odd-order FVEM) or the function value superconvergent points (for even-order FVEM).

\textbf{{\em Method I.} For odd $k$-order FVEM ($k=2l-1$)}, given $(\alpha_1,\dots,\alpha_{l-1})$, construct the dual points $g_{i,j}$ ($i=1,\dots,N$,\, $j=1,\dots,k$) accordingly. Then, the corresponding FVE scheme possesses the superconvergence of the derivative, and $g_{i,j}$ are just the derivative superconvergent points.

\textbf{{\em Method II.} For even $k$-order FVEM ($k=2l$)}, given $(l-1)$ parameters $0<\tilde{a}_{l-1}<\tilde{a}_{l-2}<\dots<\tilde{a}_1<1$
\footnote{The definitions of $\tilde{a}_j$ and $\tilde{D_j}$ are similar to $a_j$ and $D_j$ in section~\ref{subsec:notations}.
The only difference is that $\tilde{a}_j$ and $\tilde{D_j}$ are used to locate the computing nodes of the FVEM
where we put unknowns}
on the reference element $\hat{K}$, one can determine $(k+1)$ symmetric points $\tilde{D}_{j}$ ($j=0,1,\dots,k$)
(including the two endpoints and the midpoint of $\hat{K}$). Construct a $(k+1)$-order polynomial $\hat{R}_{k}$,
\begin{align*}
  \hat{R}_{k}(\xi) = \xi\, (\xi^2-1)\, \prod\limits_{j=1}^{l-1}\,(\xi^2-\tilde{a}_j^2).
\end{align*}
By the Rolle's theorem, there are $k$ different roots of $\hat{R}_{k}'=0$ on $\hat{K}$.
Denote these roots by $G_{j}$ $(j=1,2,\dots,k)$. Set the dual points accordingly,
and the corresponding FVE scheme enjoy the superconvergence of the function value.
Moreover, the points on $K\in\mathcal{T}_{h}$ corresponding to the $\tilde{D}_j$ are right the function value superconvergent points
of this FVE scheme.

\begin{rem} \label{rem:method4_verify}
It's easy to verify the resulted even-order FVE schemes from {\em Method II} satisfy the $k$-$k$-order orthogonal condition. At the same time, {\em Method II} is not valid for odd order schemes, since the $k$-$k$-order orthogonal condition can not be guaranteed for odd $k$.
\end{rem}

\begin{rem}
It's interesting to point out that, the superconvergent points of the traditional FEM are fixed and can not be freely chosen. While, we can choose the superconvergent points of the FVEM.
\end{rem}

\section{Numerical experiments} \label{sec:numerical_ex}
In this section we present several numerical results to illustrate the theoretical results in this paper. First, we present four new FVE schemes which will be used in this section.

\textbf{Scheme 3-1:} For the cubic ($k=3$) FVE scheme, $\alpha_1=\sqrt{5/9}$. The computing nodes are selected according to $\tilde{D}_0=-1,\,\, \tilde{D}_1=-1/3, \,\, \tilde{D}_2=1/3,\,\, \tilde{D}_3=1.$

\textbf{Scheme 4-1:} For the quartic ($k=4$) FVE scheme,
$(\alpha_1,\alpha_2)=(\sqrt{\frac{15+\sqrt{145}}{40}},\, \sqrt{\frac{15-\sqrt{145}}{40}})$.
This scheme is obtained by {\em Method II} with letting $\tilde{a}_1 =1/2$. That is to say,
the function value superconvergent points for this FVE scheme are selected to be uniformly arranged in each element.

\textbf{Scheme 5-1:} For the quintic ($k=5$) FVE scheme, $(\alpha_{1},\,\alpha_{2})=(\frac{\sqrt{15}}{4}, \, \frac{5\sqrt{7}}{21})$.
The parameters are obtained from (\ref{eq:alpha_relations_5order}), with taking $\alpha_{1}=\frac{\sqrt{15}}{4}$.
The computing nodes are selected according to
\[
\tilde{D}_j = \pm\sqrt{673/1344+\sqrt{459/3371}},\,\, \pm\sqrt{673/1344-\sqrt{459/3371}},\,\, \mathrm{and}\,\, \pm 1,\quad j=0,1,\dots,5.
\]

\textbf{Scheme 6-1:} For 6-order FVE scheme ($k=6$),
\[
(\alpha_1, \alpha_2, \alpha_3)\approx (0.976279142450726,\,0.637859916292150,\,0.0303474120727480).
\]
This scheme is obtained by {\em Method II} with letting $(\tilde{a}_1, \tilde{a}_2)=(\frac{19}{20},\frac{1}{19})$. The distance between computing nodes of this scheme are quite nonuniform.

\textbf{Scheme 3-1} does not satisfies
the 3-3-order orthogonal condition. While, the other three schemes satisfy the $k$-$k$-order orthogonal condition. Example~\ref{exam:Superconver_Cubic_L2} shows that \textbf{Scheme 3-1} has the superconvergence of the derivative but doesn't have the superconvergence of the function value, which helps to verify {\em Method I} and the properties for odd order schemes listed in Table~\ref{tab:relations}. Example~\ref{exam:Superconver_5order} shows that \textbf{Scheme 5-1} possesses the superconvergence of the derivative as well as the function value, which helps to verify the properties for odd order schemes listed in Table~\ref{tab:relations}.
\textbf{Scheme 4-1} and \textbf{Scheme 6-1} are both obtained by {\em Method II},
which satisfies the $k$-$(k-1)$-order orthogonal condition. Examples~\ref{exam:Superconver_4order} and \ref{exam:Superconver_6order} both help to verify the properties for even order schemes
listed in Table~\ref{tab:relations}. Moreover, since one of the dual point is quite near the end point on the righthand side
on each $K$ for \textbf{Scheme 6-1}, Example~\ref{exam:Superconver_6order} also supports that,
the dual points (derivative superconvergent points) in {\em Method I} and the computing nodes (function value superconvergent points)
in {\em Method II} can be chosen freely. In Example~\ref{exam:Superconver_4order}, figure~\ref{fig:H1L2_superconv_4order}
shows how the superconvergence phenomenon happens.

\begin{example}\label{exam:Superconver_Cubic_L2}
\begin{table}[htbp!]
  \centering
  \caption{Numerical results for Example~\ref{exam:Superconver_Cubic_L2}.}
    \begin{tabular}{ccccccccc}
    \toprule
    $h$     & $|u-u_h|_{1}$ & Order & $\|u-u_h\|_{0}$ & Order & $|u_h-u_I|_{1}$ & Order & $\|u_h-u_I\|_{0}$ & Order    \\
    \midrule
    1/2	&6.0876E-01	& \textbackslash{}	&2.8536E-02	&\textbackslash{}	&3.5633E-02	&\textbackslash{}	&4.9228E-03	&\textbackslash{}       \\
    1/4	&7.7122E-02	&2.9806 	&1.7875E-03	&3.9968 	&2.1900E-03	&4.0242 	&3.2491E-04	&3.9214\\
    1/8	&9.6579E-03	&2.9974 	&1.1165E-04	&4.0009 	&1.3632E-04	&4.0058 	&2.0814E-05	&3.9644\\
    1/16	&1.2069E-03	&3.0003 	&6.9720E-06	&4.0012 	&8.5043E-06	&4.0027 	&1.3143E-06	&3.9852\\
    1/32	&1.5081E-04	&3.0006 	&4.3551E-07	&4.0008 	&5.3100E-07	&4.0014 	&8.2522E-08	&3.9934\\
    \bottomrule
    \end{tabular}%
  \label{tab:Example3order_superconv}%
\end{table}%
We apply \textbf{Scheme 3-1} to the BVP (\ref{eq:BVP1D}) with $p(x)=2,\, q(x)=1,\, r(x)=1$,
and $f$ being chosen so that the exact solution is $u(x)=\sin x$. The first 6 columns of Table~\ref{tab:Example3order_superconv}
show that \textbf{Scheme 3-1} has the optimal $H^1$ and $L^2$ convergence rate as well as the superconvergence of the derivative.
While, the last two columns of Table~\ref{tab:Example3order_superconv} indicate the function value of $u_h$ is not superclose to $u_I$.

Of course, we can not simply conclude that the corresponding FVE scheme does not possess superconvergence property of the function value,
because the choice of $u_I$, which may affects the numerical results, is not unique. In other words,
the $k$-$k$-order orthogonal condition is sufficient conditions for the superconvergence of the function value of the FVEM.
\end{example}

\begin{example}\label{exam:Superconver_5order}
\begin{table}[htbp!]
  \centering
  \caption{Numerical results for Example~\ref{exam:Superconver_5order}.}
    \begin{tabular}{ccccccccc}
    \toprule
    $h$     & $|u-u_h|_{1}$ & Order & $\|u-u_h\|_{0}$ & Order & $|u_h-u_I|_{1}$ & Order & $\|u_h-u_I\|_{0}$ & Order    \\
    \midrule
    1/2	&6.2039E-03	&\textbackslash{}&	3.1708E-04	&\textbackslash{}&	2.1126E-04	&\textbackslash{}&	1.5302E-05	&\textbackslash{}    \\
    1/3	&8.2572E-04	&4.9738 	&2.8098E-05	&5.9770 	&1.8541E-05	&6.0008 	&8.8949E-07	&7.0169 \\
    1/4	&1.9663E-04	&4.9880 	&5.0159E-06	&5.9896 	&3.2925E-06	&6.0078 	&1.1826E-07	&7.0139 \\
    1/5	&6.4527E-05	&4.9933 	&1.3166E-06	&5.9943 	&8.6207E-07	&6.0054 	&2.4754E-08	&7.0083 \\
    1/6	&2.5952E-05	&4.9958 	&4.4120E-07	&5.9965 	&2.8849E-07	&6.0040 	&6.9031E-09	&7.0042 \\
    \bottomrule
    \end{tabular}%
  \label{tab:Example5order_superconv}%
\end{table}%
We apply \textbf{Scheme 5-1} to the BVP (\ref{eq:BVP1D}) with $p(x)=e^x,\, q(x)=\sin x,\, r(x)=3$,
and $f$ being chosen so that the exact solution is $u(x)=\sin x$. Table~\ref{tab:Example5order_superconv} shows \textbf{Scheme 5-1}
possesses all the four properties listed in Table~\ref{tab:relations}. The results verify that,
if the $5$-$5$-order orthogonal condition is satisfied, the corresponding FVE scheme has the superconvergence of the function value.
Moreover, the dual points are just the derivative superconvergent points, and the 6 function value superconvergent points in each element can be derived by {\em Method II}.
\end{example}

\begin{example}\label{exam:Superconver_4order}
\begin{figure}[!ht]
    \centering
    \subfigure{
    \begin{minipage}[t]{.43\textwidth}
      \centering
      \includegraphics*[width=200pt]{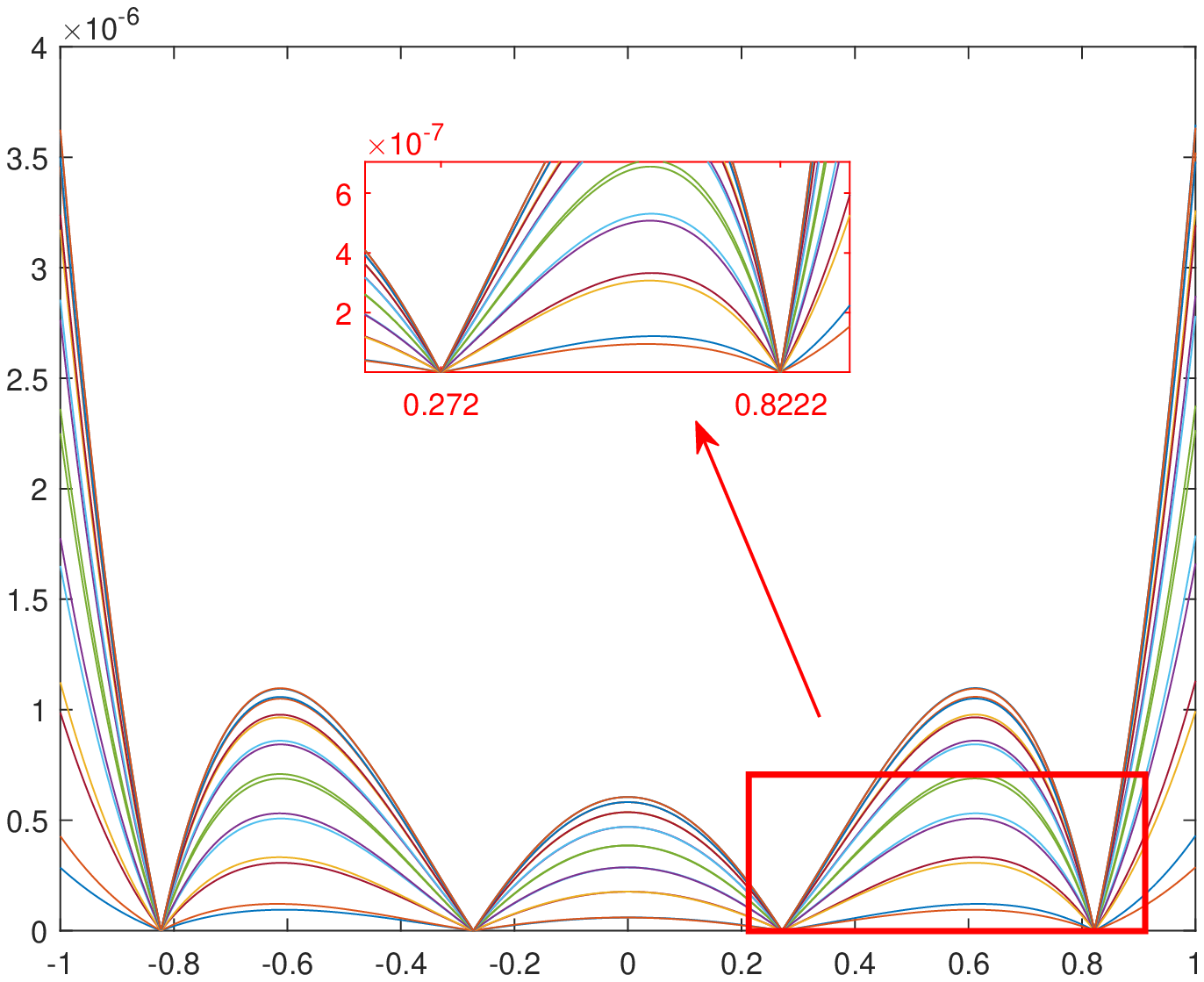}
      \title{(a) $|u'-u_{h}'|$ on each element $K$ been projected to the reference element $\hat{K}$.}
    \end{minipage}}
    \subfigure{
    \begin{minipage}[t]{.43\textwidth}
      \centering
      \includegraphics[width=200pt]{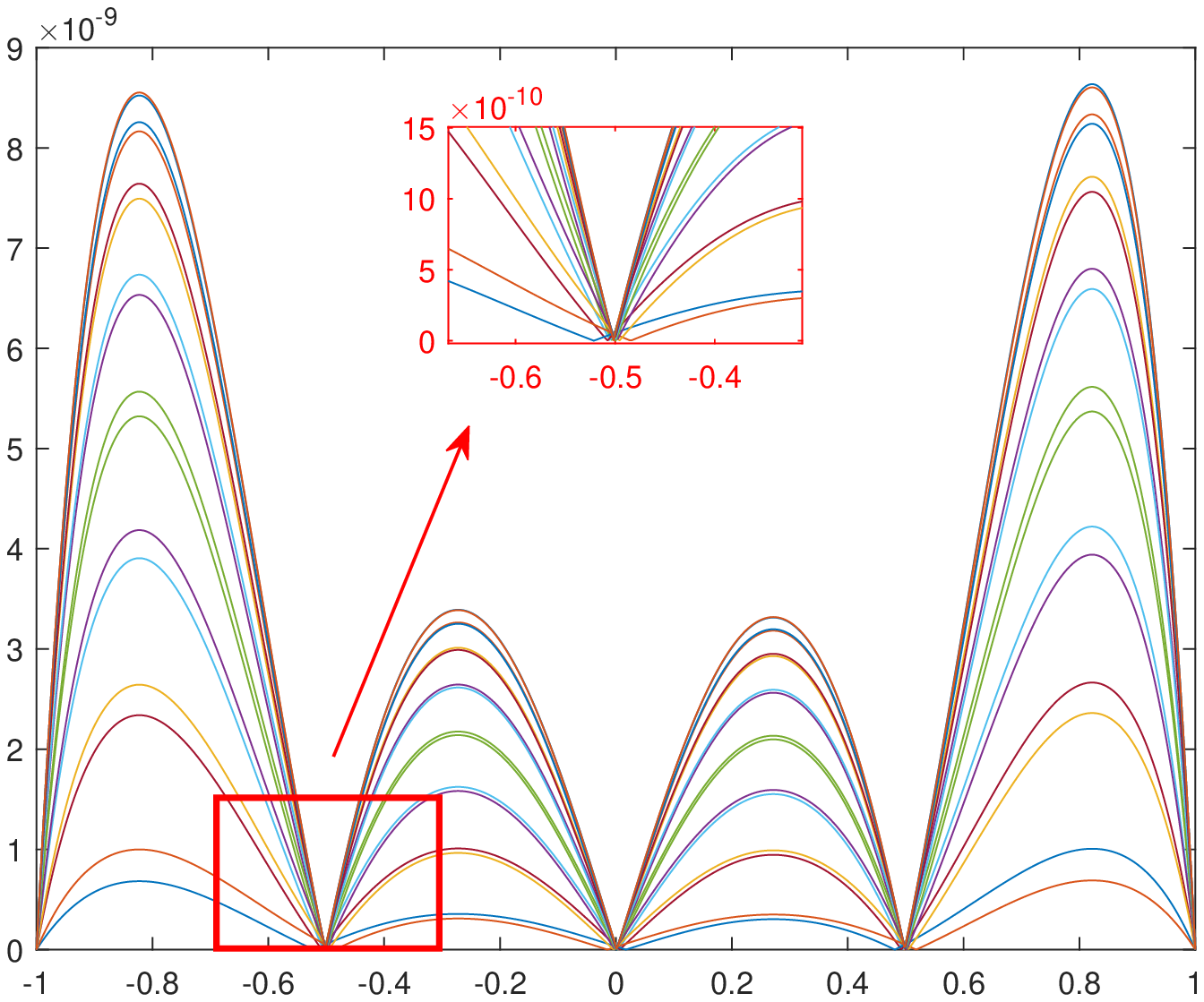}
      \title{(b) $|u-u_{h}|$ on each element $K$ been projected to the reference element $\hat{K}$.}
    \end{minipage}}
    \caption{Numerical performance of \textbf{Scheme 4-1} applied to Example~\ref{exam:Superconver_4order}.
    The mesh size is $h=1/16$. The errors on each element $K$ are shown together on the reference element $\hat{K}=[-1,1]$.}
    \label{fig:H1L2_superconv_4order}
\end{figure}

\begin{table}[htbp!]
  \centering
  \caption{Numerical solution of Example~\ref{exam:Superconver_4order}.}
    \begin{tabular}{ccccccccc}
    \toprule
    $h$     & $|u-u_h|_{1}$ & Order & $\|u-u_h\|_{0}$ & Order & $|u_h-u_I|_{1}$ & Order & $\|u_h-u_I\|_{0}$ & Order    \\
    \midrule
    1/2&	6.5281E-02&	\textbackslash{}   &2.3887E-03	& \textbackslash{}	&2.5868E-03&	\textbackslash{}&	2.3297E-04&	\textbackslash{}\\
    1/4	&4.1077E-03	&3.9902     &7.4542E-05	&5.0020 	&8.3620E-05	&4.9512 	&3.6653E-06	&5.9901\\
    1/8	&2.5656E-04	&4.0010 	&2.3212E-06	&5.0051 	&2.6414E-06	&4.9845 	&5.7499E-08	&5.9943\\
    1/16&1.6014E-05	&4.0019 	&7.2357E-08	&5.0036 	&8.2871E-08	&4.9943 	&9.0071E-10	&5.9963\\
    \bottomrule
    \end{tabular}%
  \label{tab:Example4order_superconv}%
\end{table}%
Consider the BVP (\ref{eq:BVP1D}) with $p(x)=2,\, q(x)=1,\, r(x)=1$, and $f$ being chosen so that the exact solution
of this problem is $u(x)=\sin x$. Table~\ref{tab:Example4order_superconv} indicates that \textbf{Scheme 4-1} has
optimal $H^1$ and $L^2$ convergence rate, as well as the superconvergence of the derivative and the function value.
Figure~\ref{fig:H1L2_superconv_4order} shows the errors of the derivatives $|u'-u_{h}'|$ (subfigure (a)) and the function values $|u-u_{h}|$ (subfigure (b)) on each element $K$, mapped to the reference element $\hat{K}=[-1,1]$ together. We can see that the high accuracy points (also the superconvergent points) of $u_h'-u'$ and $u_h-u$ well fit our theoretical results (since $(\sqrt{(15+\sqrt{145})/40},\, \sqrt{(15-\sqrt{145})/40})\approx (0.8222,0.2720)$).
\end{example}

\begin{example}\label{exam:Superconver_6order}
\begin{table}[htbp!]
  \centering
  \caption{Numerical solution of Example~\ref{exam:Superconver_6order}.}
    \begin{tabular}{ccccccccc}
    \toprule
    $h$     & $|u-u_h|_{1}$ & Order & $\|u-u_h\|_{0}$ & Order & $|u_h-u_I|_{1}$ & Order & $\|u_h-u_I\|_{0}$ & Order    \\
    \midrule
    1/2	&4.3272E-04	&\textbackslash{}	&2.5266E-05	&\textbackslash{}	&1.7730E-05	&\textbackslash{}&	2.6053E-06	&\textbackslash{}          \\
    1/3	&3.8239E-05	&5.9839 	&1.4806E-06	&6.9970 	&1.0594E-06	&6.9490 	&1.0317E-07	&7.9636 \\
    1/4	&6.8169E-06	&5.9943 	&1.9760E-07	&7.0006 	&1.4092E-07	&7.0120 	&9.9423E-09	&8.1324 \\
    1/5	&1.7879E-06	&5.9978 	&4.1606E-08	&6.9821 	&2.7779E-08	&7.2776 	&1.6297E-09	&8.1042 \\
    \bottomrule
    \end{tabular}%
  \label{tab:Example6order_superconv}%
\end{table}%
We apply \textbf{Scheme 6-1} to the same problem of Example~\ref{exam:Superconver_4order}.
Table~\ref{tab:Example6order_superconv} shows this scheme has all the four properties listed in Table~\ref{tab:relations}.
\end{example}


%
\section{Conclusion}  \label{sec:conclusion}
In this paper, new superconvergent structures are developed from the FVEM,
which includes the Gauss-Lobatto structure and covers much more FVE schemes than the Gauss-Lobatto structure.
By proposing the more general $k$-$r$-order ($k-1\leq r\leq 2k-2$) orthogonal condition and the modified M-decomposition (MMD),
we prove the superconvergence properties for the FVE schemes which satisfy the $k$-$(k-1)$-order (superconvergence of the derivative)
and the $k$-$k$-order (superconvergence of the function value) orthogonal condition.

Easy ways to construct the FVE schemes are presented in subsection~\ref{subsec:construction_easy} ({\em Method I} and {\em Method II}). For odd $k$-order FVEM, we can freely choose the $k$ symmetric derivative
superconvergent points of a FVE scheme on primary element $K$ (excluding the 2 end points of $K$); for even $k$-order FVEM, we can freely choose any $(k+1)$ symmetric function value superconvergent points of a FVE scheme on primary element $K$ (including the 2 end points of $K$). These facts provide us more freedom of choosing the superconvergent points.

In addition, all FVE schemes over symmetric dual meshes are proved to be unconditionally stable. And, the relationships between the orthogonal condition with the convergence properties of the FVE schemes are figured out in table~\ref{tab:relations}: all FVE schemes holds optimal $H^1$ estimate;
the $k$-$(k-1)$-order orthogonal condition ensures the superconvergence of the derivative and optimal $L^2$ estimate; the $k$-$k$-order orthogonal condition ensures the superconvergence of the function values.
Numerical experiments confirm our theoretical results.

The extension of the work at hand to the 2D case is the our next step ongoing. The ideas and methods developed here are instructive to 2D problems on rectangular meshes, while the theory in 2D is not straightforward.

\section{Appendix A: stability and H1 estimate}
\label{sec:stability}
The stability and $H^1$ estimate are the issues we can not skip when we study the $L^2$ estimate and superconvergence.
The authors of \cite{Cao.2013,Plexousakis.2004} gave some results for FVE schemes with some special dual strategies,
such as the Gauss-Lobatto FVE schemes. In this section, we prove the stability and $H^1$ estimate for general FVE schemes
with symmetric dual meshes. The proof in this section benefits a lot from the $k$-points numerical quadrature and \cite{Cao.2013}.




We begin with some notations specially used in this section. Firstly, for all $w\in H_{\mathcal{T}}^{m}(\Omega):=\{ w\in C(\Omega):\, w|_{K_{i}} \in H^{m}(K_{i}), \, \forall K_{i}\in\mathcal{T}_{h} \},$ and all $j\geq0$, we define a semi-norm and a norm by
    \begin{align}  \label{eq:norm_Sobolev}
    |w|_{j,\mathcal{T}}=\left(\sum_{ K_{i}\in\mathcal{T}_{h} } |w|_{j,K_{i}}^{2} \right)^{\frac{1}{2}},
    \qquad
    \|w\|_{m,\mathcal{T}}=\left(\sum_{ j=0 }^{m} |w|_{j,\mathcal{T} }^{2} \right)^{\frac{1}{2}}.
        \end{align}


Secondly, for all $v_{h}=\sum\limits_{i=1}^{N}\sum\limits_{j=1}^{k} v_{i,j} \psi_{i,j} \in V_h$, let
    \begin{align}  \label{eq:norm_dual}
    &|v_{h}|^{2}_{1,\mathcal{T}_{h}^{*}} = \sum_{i=1}^{N}\sum_{j=1}^{k} h_i^{-1}[v_{i,j}]^{2},
    \qquad \|v_{h}\|^{2}_{0,\mathcal{T}_{h}^{*}} = \sum_{i=1}^{N}\sum_{j=1}^{k} h_i v_{i,j}^{2}, \\
    &\| v_{h}\|^{2}_{1,\mathcal{T}_{h}^{*}} =
    | v_{h}|^{2}_{1,\mathcal{T}_{h}^{*}}+\|v_{h}\|^{2}_{0,\mathcal{T}_{h}^{*} }.
    \end{align}
    Noticing that $v_{1,0} = v_{N,k}=0$, the following Poincar\'e inequality holds naturally
    \[
    \| v_{h}\|_{0,\mathcal{T}_{h}^{*}} \leq C |v_{h}|_{1,\mathcal{T}_{h}^{*}}, \quad
    \forall v_{h}\in V_{h},
    \]
    where the constant $C$ depends only on $\Omega$ and $k$.

Thirdly, we denote $A_j$ ($j\in \mathbb{Z}_{k}$) the weights of the $k$-points numerical quadrature
    $Q_k(F) = \sum_{j=1}^{k} A_{j} F(G_{j})$
    for computing the integral $I(F) = \int_{-1}^{1} F(x) \ud x.$
    Naturally, the weights on interval
    $K_{i}$ $(i\in \mathbb{Z}_{N})$ are
    $A_{i,j}=\frac{h_i}{2}A_{j}$ $(j\in \mathbb{Z}_{k}).$
    Then, we define a discrete semi-norm $|\cdot|_{1,G}$ by
    \begin{align}  \label{eq:norm_discrete}
    |w|_{1,G} = \left( \sum_{i=1}^{N} \sum_{j=1}^{k} A_{i,j} (w'(g_{i,j}))^2 \right)^{\frac{1}{2}},
    \quad \forall w\in H_{0}^{1}.
    \end{align}

Fourthly, a linear mapping  $\Pi_{\mathcal{T}}^{*}: U_{h}^{k}\rightarrow V_{h}$ is given by ($\Pi_{\mathcal{T}}^{*}$ is different from $\Pi_{h}^{k,*}$ defined in subsection~\ref{subsec:notations}, and $\Pi_{\mathcal{T}}^{*}$ will be used only in this section)
    \begin{align}\label{eq:PiT}
        \Pi_{\mathcal{T}}^{*} w_{h} = \sum_{i=1}^{N}\sum_{j=1}^{k} w_{i,j} \psi_{i,j},
        \quad w_{h}\in U_{h}^{k},
    \end{align}
    where the coefficients $w_{i,j}$ are determined by the constraints
    $[w_{i,j}] = A_{i,j}w_{h}'(g_{i,j})$, $((i,j)\in\mathbb{Z}_{N} \times \mathbb{Z}_{k} \backslash \{(N,k)\})$.
    Similar with \cite{Cao.2013}, we also have $[w_{N,k}] = w_{N,k} - w_{N,k-1} = A_{N,k} w_{h}'(g_{N,k})$.

According to the idea of the proof of \cite{Cao.2013}, with the help of the $k$-points quadrature, we present the following lemma without the details of the proof.
\begin{lem}\label{lem:norm_equivalence}
Given an FVE scheme, the semi-norms given by $(\ref{eq:norm_dual})$, $(\ref{eq:norm_discrete})$ and $(\ref{eq:norm_Sobolev})$ are equivalent.
\[
|\Pi_{\mathcal{T}}^{*} w_{h}|_{1,\mathcal{T}_{h}^{*}}
\sim |w_{h}|_{1,G}
\sim |w_{h}|_{1,\mathcal{T}}, \quad \forall w_{h} \in U_{h}^{k}.
\]
\end{lem}

\begin{thm} \label{thm:infsup}
For sufficiently small the mesh size $h$, the following inf-sup condition are satisfied.
\begin{align} \label{eq:infsup}
\inf_{w_{h}\in U_{h}^{k}}  \sup_{v_{h}\in V_{h}} \frac{a_{h}(w_{h},v_{h})}{\|w_{h}\|_{1} \|v_{h}\|_{\mathcal{T}_h^{*}}} \geq c_{0},
\end{align}
where $c_{0}>0$ is a constant depending only on $k$, $\alpha_{0}$, $\kappa$ and $\Omega$.
\end{thm}

\begin{proof}
It follows from the bilinear form (\ref{eq:FVEM_BVP1D}) that
\[
a_{h}(w_{h}, \Pi_{\mathcal{T}}^{*} w_{h}) = I_{1} +I_{2}, \quad \forall w_{h}\in U_{h}^{k},
\]
with
\[
I_{1} = \sum_{i=1}^{N} \sum_{j=1}^{k} [w_{i,j}] p(g_{i,j})w_{h}'(g_{i,j}), \quad
I_{2} = \sum_{i=1}^{N} \sum_{j=1}^{k} w_{i,j} \int_{g_{i,j}}^{g_{i,j+1}} (q(x)w_{h}'(x)+r(x)w_{h}(x)) \, \ud x
\]
Therefore,
\[
I_{1} \geq p_{0} \sum_{i=1}^{N} \sum_{j=1}^{k} A_{i,j} (w_{h}'(g_{i,j}))^2  \sim  p_{0} |w_{h}|_{1}^2.
\]

Let $V(x)=\int_{a}^{s} \,(q(s)w_{h}'(s)+r(s)w_{h}(s) )\, d s$ and denote by
\[
E_{i} = \int_{x_{i-1}}^{x_{i}} w_{h}'(x)\,V(x) \, d x - \sum_{j=1}^{k} A_{i,j} w_{h}'(g_{i,j})V(g_{i,j})
\]
the error of the $k$-points numerical quadrature in the interval $[x_{i-1}, x_{i}]$, $i\in \mathbb{Z}_{N}$. Then
\[
I_{2} = -\sum_{i=1}^{N} \sum_{j=1}^{k} [w_{i,j}] V(g_{i,j}) = -\int_{a}^{b} \, w_{h}'(x) \, V(x) \,\ud x + \sum_{i=1}^{N} \, E_{i}.
\]
With the fact that $w_{h}(a)=w_{h}(b)=0$ and
\[
\int_{a}^{b} \, q(x)\,w_{h}'(x)\,w_{h}(x)\, \ud x = -\frac{1}{2}\int_{a}^{b} \, q'(x)\,w_{h}^{\,2}(x) \, \ud x,
\]
we obtain
\[
-\int_{a}^{b} \, w_{h}'(x) \, V(x) \,\ud x = \int_{a}^{b} \, (r(x)-\frac{q'(x)}{2}) \, w_{h}^{\,2}(x) \, \ud x
\geq \gamma \|w_{h}\|_{0}^{2}.
\]
On the other hand, by \cite{Davis.1984}, for all $i\in\mathbb{Z}_{N}$
\[
E_{i} = (w_{h}'V)^{(k)}(\xi_{i}) \, O(h_{i}^{k+1}),
\]
where $\xi_{i}\in[x_{i-1},x_{i}]$. By the Leibnitz formula of  derivatives and the inverse inequality, we have
\begin{align*}
|(w_{h}'V)^{(k)}(\xi_{i})|  \leq &  \sum_{t=1}^{k}
\left(
\begin{array}{c}
   k \\
   t \\
\end{array}
\right)
|(qw_{h}' + rw_{h})^{(t-1)} (w_{h}')^{(k-t)}(\xi_{i})|    \\                                                 \leq &  c_{1} \,
\sum_{t=1}^{k} \, \|w_{h}\|_{t,\infty,K_{i}} \, \|w_{h}\|_{(k-t+1),\infty,K_{i}}   \\
\lesssim &  c_{1} \,
\sum_{t=1}^{k}
\, h^{-(t-\frac{1}{2})} \, |w_{h}|_{1,K_{i}} \, h^{-(k-t+1-\frac{1}{2})}\, |w_{h}|_{1,K_{i}}  \\
= &  \tilde{c}_{1}\,
\, h^{-k}\, |w_{h}|_{1,K_{i}}^{2},
\end{align*}
with
\[
c_{1}=\max \{\|q\|_{k-1,\infty,K_{i}},\, \|r\|_{k-1,\infty,K_{i}}\} \max_{t\leq k}\left(
\begin{array}{c}
   k \\
   t \\
\end{array}
\right).
\]
Combining the above estimates, we have
\[
I_{2} \gtrsim \gamma \|w_{h}\|_{0,\mathcal{T}}^{2} - \tilde{c}_{1} h_{i}|w_{h}|_{1,\mathcal{T}}^{2},
\]
where $\tilde{c}_{1}$ is a constant independent of $h_i$. Then for sufficiently small $h$, we have
\[
a_{h}(w_{h}, \Pi_{\mathcal{T}}^{*} w_{h})
\geq \frac{p_{0}}{2}|w_{h}|_{1,\mathcal{T}}^{2} + \frac{\gamma}{2}\|w_{h}\|_{0,\mathcal{T}}^{2}
\geq \frac{1}{2} \min\{p_{0},  \gamma \} \|w_{h}\|_{1,\mathcal{T}}^{2}.
\]
By Lemma~\ref{lem:norm_equivalence}, one has
\[
\|w_{h}\|_{1,\mathcal{T}} \gtrsim \|\Pi_{\mathcal{T}}^{*} w_{h}\|_{\mathcal{T}_h^{*}}.
\]
Therefore, for any $w_{h}\in U_{h}^{k}$, we can obtain
\[
\sup_{v_{h}\in V_{h}} \frac{a_{h}(w_{h},\, v_{h})}{ \| v_{h} \|_{\mathcal{T}_h^{*}} }
\gtrsim \frac{a_{h}(w_{h},\, \Pi_{\mathcal{T}}^{*} w_{h} )}{ \| \Pi_{\mathcal{T}}^{*} w_{h} \|_{\mathcal{T}_h^{*}} } \geq c_{0} \|w_{h}\|_{1,\mathcal{T}},
\]
where $c_{0}$ is a constant depending only on $k$, $p_{0}$, $\gamma$, and $\Omega$, the inf-sup condition (\ref{eq:infsup}) then follows.
\end{proof}

With the inf-sup condition (\ref{eq:infsup}) and a similar procedure to \cite{Cao.2013},
we have the $H^1$ estimate for FVE schemes with symmetric dual meshes.
\begin{thm} \label{thm:H1}
Assume that $u\in H_{0}^{1}(\Omega)\cap H_{\mathcal{T}}^{k+1}(\Omega)$ is the solution of (\ref{eq:BVP1D}),
and $u_{h}$ is the FVE solution of (\ref{eq:FVEM_BVP1D}). Then we have
\begin{align}\label{eq:H1estimate}
\|u-u_{h}\|_{1} \leq Ch^{k} \|u\|_{k+1,\mathcal{T}},
\end{align}
where $C$ is a constant independent of $h$.
\end{thm}


\bibliographystyle{abbrv} 



\bibliography{Framework20190909}

\end{document}